\documentclass[11pt]{amsart}
\usepackage{amssymb}
\usepackage{amsmath, amsfonts,enumerate}
\usepackage{hyperref}
\usepackage{euscript}

\let\cal\mathcal

\makeatletter \@mparswitchfalse \makeatother
\newtheorem{theorem}{Theorem}
\newtheorem{lemma}[theorem]{Lemma}
\newtheorem{sublemma}[theorem]{Sublemma}
\newtheorem{corollary}[theorem]{Corollary}
\newtheorem{proposition}[theorem]{Proposition}
\theoremstyle{remark}
\newtheorem{remark}[theorem]{Remark}

\theoremstyle{definition}
\newtheorem{definition}[theorem]{Definition}
\newtheorem{problem}[theorem]{Problem}

\numberwithin{equation}{section}
\numberwithin{theorem}{section}

\def\M{\cal{M}}

\let\<\langle
\def\ch{\raise 0.5ex \hbox{$\chi$}}
\def\T{\tau}
\def\E{\cal{E}}

\let\\\cr
\let\phi\varphi

\let\epsilon\varepsilon

\def\log{\operatorname{log}}

\renewcommand{\i}{{\rm i}}

\newcommand{\N}{\cal{N}}
\newcommand{\h}{\mathsf{h}}

\newcommand{\Tr}{\mbox{\rm Tr}}
\newcommand{\tr}{\mbox{\rm tr}}

\begin{document}

\title[Burkholder inequalities]{Noncommutative Burkholder/Rosenthal  inequalities associated with convex functions}
\author[N. Randrianantoanina]{Narcisse Randrianantoanina}
\address{Department of Mathematics, Miami University, Oxford,
Ohio 45056, USA}
 \email{randrin@miamioh.edu}

\author[L. Wu]{Lian Wu}
 \address{School of Mathematics and Statistics, Central South University, Changsha 410085, China and  Department of Mathematics, Miami University, Oxford, Ohio 45056, USA}
 
 \date{\today}
 
 \thanks{Wu was partially supported by NSFC(No.11471337) and  the China Scholarship Council}
 \email{wul5@miamioh.edu}
 
\subjclass[2010]{Primary: 46L53, 46L52. Secondary: 47L05, 60G42}
\keywords{Noncommutative Burkholder inequalities, Noncommutative Rosenthal inequalities, Orlicz functions, moment inequalities, interpolations}

\begin{abstract} 
We prove  noncommutative martingale inequalities associated with convex functions. More precisely,  we obtain $\Phi$-moment analogues of the noncommutative Burkholder inequalities and the noncommutative Rosenthal inequalities  for any convex Orlicz function $\Phi$ whose  Matuzewska-Orlicz indices  $p_\Phi$ and $q_\Phi$ are such that  $1<p_\Phi\leq q_\Phi <2$ or $2<p_\Phi \leq q_\Phi<\infty$. These results generalize the noncommutative Burkholder/Rosenthal inequalities due to Junge and Xu.
 \end{abstract}

\maketitle


\section{Introduction}
 The theory of noncommutative martingales has enjoyed  considerable progress in recent years  due to  its interaction with other field of mathematics such as operator spaces and free probability. Many  classical  martingale inequalities have been extended to the noncommutative setting.  We refer to \cite{Junge-Perrin, JX, JX2, PX, Ran15}  and the references therein for more information on noncommutative martingales. This paper deals with  moment inequalities  associated with convex functions for noncommutative martingales.

The study of convex function inequalities for martingales 
was initiated by Burkholder and Gundy in  their seminal  paper \cite{BG}.
The general theme of their work  can be summarized as follows: let  $\mathfrak{M}$
be a family of martingales on a probability space $(\Omega, \Sigma, \mathbb{P})$ 
and $\Phi$ be a nonnegative  and increasing convex function  on $[0,\infty)$. If $U$ and $V$ are operators on $\mathfrak{M}$ with values in the set of nonnegative random variables on $(\Omega, \Sigma, \mathbb{P})$, under what conditions on $\Phi$  and $\mathfrak{M}$ does the inequality $\mathbb{E}\big[ \Phi(Vf)\big] \leq C\mathbb{E}\big[ \Phi(Uf)\big]$ hold for all  martingales $f \in \mathfrak{M}$.
For the  special case where $\Phi(t)=t^p$  for $1\leq p<\infty$, the above question reduces to comparisons of $p$-th moments of  the nonnegative  random variables $Vf$ and $Uf$. For general  convex function $\Phi$, these types of inequalities are  generally referred to as $\Phi$-moment inequalities. Typical examples of such operators $U$ and $V$ are, among others, square functions,  maximal functions, martingale transforms, ect. Subsequently,  many classical  $p$-th moment inequalities for martingales were extended  to convex  function inequalities.  We refer to \cite{Bu1, Burkholder-Davis-Gundy, Garsia2}  for more information on  the development of $\Phi$-moment inequalities  from the classical martingale theory.

 Recently,  several $\Phi$-moment inequalities  have been  extended to the context of noncommutative martingales. This was initiated by Bekjan and Chen in \cite{Bekjan-Chen}. For instance,   $\Phi$-moment versions of  the noncommutative Burkholder-Gundy inequalities from \cite{PX}  were considered in \cite{Bekjan-Chen, Dirksen-Ricard}. Various maximal  type-inequalities for noncommutative martingales initially proved  in \cite{Ju}  for the case of noncommutative $L_p$-spaces are now known to be valid for  a wider class  of convex functions
(\cite{Bekjan-Chen-Ose, Dirksen}). In this paper, we are mainly interested on inequalities involving conditioned square functions of noncommutative martingales.  To better explain  our motivation  and  results,  let us begin by recalling Rosenthal's remarkable inequalities (\cite{Ros}) which state  that if $2\leq p<\infty$ and $(g_n)_{n\geq 1}$ is a sequence of independent mean-zero random variables  in $L_p(\Omega, \Sigma, \mathbb{P})$, then the following holds:
\begin{equation}\label{Rosenthal}
\Big(\mathbb{E}\big| \sum_{n\geq 1} g_n \big|^p\Big)^{1/p} \simeq_p \Big( \sum_{n\geq 1} \mathbb{E}|g_n|^2\Big)^{1/2}  + \Big( \sum_{n\geq 1} \mathbb{E}|g_n|^p \Big)^{1/p},
\end{equation}
where $\simeq_p$ means equivalence of norms up to constants depending only on $p$. The equivalence  \eqref{Rosenthal} was initially established  in order to construct some new classes of Banach space  but over the years it has  been proven to have 
many  applications in other areas of mathematics. The martingale version of  \eqref{Rosenthal}  was discovered  almost simultaneously by Burkholder in \cite{Bu1}. In fact,  a $\Phi$-moment version was obtained by Burkholder that takes   the following form: if $\Phi$ is a convex Orlicz function on $[0,\infty)$ that satisfies the so called $\Delta_2$-condition then  for any martingale $f=(f_n)_{n\geq 1}$  adapted to a given  filtration $\{\Sigma_n\}_{n\geq 1}$ of $\sigma$-subalgebras  of $\Sigma$ satisfying $\sigma(\bigcup_{n\geq 1} \Sigma_n)=\Sigma$,   the following holds (here, we use the convention that $\Sigma_0=\Sigma_1$):

 \begin{equation}\label{Burkholder}
\sup_{n\geq 1} \mathbb{E}\big[\Phi (|f_n |)\big] \leq C_\Phi \mathbb{E}\big[ \Phi(s(f))\big] + \mathbb{E}\big[ \Phi(d^*)\big],
\end{equation}
where $s(f)=\big(\sum_{n\geq 1} \mathbb{E}[|df_n|^2|\Sigma_{n-1}] \big)^{1/2}$ is  the conditioned square function of  the martingale $f$ while  $d^*=\sup_{n\geq 1} |df_n|$ is the maximal function of its martingale difference sequence.
On the other hand,  noncommutative analogues of the Burkholder/Rosenthal inequalities 
for the case of $p$-th moments have been discovered by Junge and Xu in \cite{JX,JX3}. More precisely, they  obtained that if  $2\leq p <\infty$ and $x=(x_n)_{n\geq 1}$ is a noncommutative martingale that is $L_p$-bounded  then
\begin{equation}\label{nc1}
\big\|x\big\|_p \simeq_p \max\Big\{ \big\|s_c(x)\big\|_p,  \big\|s_r(x)\big\|_p, \big( \sum_{n\geq 1} \big\|dx_n\big\|_p^p \big)^{1/p}\Big\}
\end{equation}
where $s_c(x)$ and $s_r(x)$ denote the column version and the row version  of  conditioned square functions which we refer to the next section for formal definitions. In addition, they also managed to formulate and prove   the corresponding inequalities for  the range $1<p<2$ which are  dual to  \eqref{nc1} that  can be roughly stated as follows: if $x=(x_n)_{n\geq 1}$ is a noncommutative martingale in $L_2(\M)$ then 
\begin{equation}\label{nc2}
\big\|x \big\|_p \simeq_p \inf\Big\{ \big\|s_c(y)\big\|_p +  \big\|s_r(z)\big\|_p + \big( \sum_{n\geq 1} \big\|dw_n\big\|_p^p \big)^{1/p}\Big\}
\end{equation}
where  the infimum is taken over all $x=y+z +w$  with $y$, $z$, and $w$ are martingales. Reasons behind the fact that  the two cases $1<p<2$ and $2\leq p<\infty$ have to be different are now well-understood in the field. As shown in \cite{JX,JX3}, the equivalences \eqref{nc1} and \eqref{nc2} have far reaching applications ranging from random matrices to operator space classifications of some classes of subspaces of noncomutative $L_p$-spaces.
Recently, equivalences \eqref{nc1} and \eqref{nc2} were extended  to certain classes of noncommutative symmetric spaces for which we refer to \cite{RW} for details. 
 Motivated by these various results, we consider $\Phi$-moments of conditioned square functions of noncommutative martingales  in the spirit of \eqref{Burkholder}. 
 We obtain  natural extensions of the noncommutative Burkholder inequalities \eqref{nc1} and \eqref{nc2}. We work with semifinite von Neumann algebra equipped with normal semifinite faithful  trace $(\M,\T)$.  In formulating the right versions of $\Phi$-moments, one needs to consider martingales that are bounded in the noncommutative Orlicz space $L_\Phi(\M)$. In addition, we also require some conditions on the lower and upper Matuzewska-Orlicz indices  $p_\Phi$ and $q_\Phi$ of the convex function $\Phi$ which in some sense mimic  the role of the index $p$ in the  noncommutative Burkholder/Rosenthal inequalities.
 Our principal results may be viewed as common generalizations of \eqref{Burkholder}, \eqref{nc1}, and \eqref{nc2}. We may summarize  these results  as follow:
 
If $2<p_\Phi \leq q_\Phi<\infty$, then  for any $L_\Phi(\M)$-bounded martingale $x=(x_n)_{n\geq 1}$, 
 \begin{equation}\label{nc3}
\sup_{n\geq 1}\T\big[ \Phi\big(|x_n|\big)\big] \simeq_\Phi \max\Big\{ \T\big[ \Phi\big(s_c(x)\big)\big], \T\big[ \Phi\big(s_r(x)\big)\big] ,  \sum_{n\geq 1} \T\big[ \Phi\big(|dx_n|\big)\big] \Big\}.
\end{equation}

If $1<p_\Phi \leq q_\Phi<2$, then for any $L_\Phi(\M)$-bounded martingale $x=(x_n)_{n\geq 1}$, 
 \begin{equation}\label{nc4}
\sup_{n\geq 1}\T\big[ \Phi\big(|x_n|\big)\big] \simeq_\Phi \inf\Big\{ \T\big[ \Phi\big(s_c(y)\big)\big]+ \T\big[ \Phi\big(s_r(z)\big)\big] + \sum_{n\geq 1} \T\big[ \Phi\big(|dw_n|\big)\big] \Big\}
\end{equation}
where the infimum is taken over all $x=y+z+w$ with $y$, $z$, and $w$ are martingales. We refer to Theorem~\ref{main} and Theorem~\ref{pl2} for more detailed explanations of the notation used in the formulations of \eqref{nc3} and \eqref{nc4}.
These results complement the series  of $\Phi$-moment inequalities from \cite{Bekjan-Chen, Bekjan-Chen-Ose, Dirksen, Dirksen-Ricard}. We note that if $\Phi(t)=t^p$  for $1<p<\infty$, then these results become  exactly the  Junge and Xu's noncommutative Burkholder inequalities.  It is also important to note that the case of noncommutative symmetric spaces treated in \cite{RW} does not  imply the corresponding $\Phi$-moment  inequalities.

The original proof  of \eqref{Burkholder}  was primarily based on careful analysis of distribution functions which heavily relied  on  stopping times and  the so-called good $\lambda$-inequalities. Stopping times and good $\lambda$-inequalities  are very powerful  techniques in the classical setting.
Unfortunately, these techniques are not available in the noncommutative  setting. Therefore,   
our method of proof  has to rely on new ideas. Our approach  was primarily motivated by   an observation that singular values of measurable operators   are closely connected to $K$-functionals from interpolation theory.  Our strategy is to focus first on  \eqref{nc4}. As noted earlier,  we heavily employ results from interpolation theory. As in the case of noncommutative symmetric spaces, a simultaneous  decomposition version of \eqref{nc2}  from \cite{RW} also plays a significant role in our argument.  
The proof of  \eqref{nc3}  is a duality type-argument.    Since $\Phi$-moments are usually not defining a norm,  we had to provide  the proper connection between  any  given  Orlicz function and its complementary that is suitable for moment inequalities. This  connection appears as  an operator  equality   that may be viewed as operator reverse  to the classical 
 Young's inequality. We refer   to  Proposition~\ref{duality}  for the exact statement. We should point out that  for the case  of square functions, the proofs  of the $\Phi$-moment versions of the noncommutative Burkholder-Gundy in \cite{Bekjan-Chen, Dirksen-Ricard}  depend  on  some versions of $\Phi$-moment extensions  of  the noncommutative Khintchine inequalities.
 %
 
The paper is structured as follows. In Section~2, we  setup   some basic notation  and present some preliminary results 
 concerning  noncommutative Orlicz spaces and noncommutative martingales.  We review the constructions leading up to   all relevant Hardy type spaces that we need for our presentation. In Section~3, we isolate and prove some key inequalities involving  $\Phi$-moments, $K$-functionals, and $J$-functionals from interpolation theory. Section~4 is devoted to  the statements and proofs of our $\Phi$-moment versions of the noncommutative Burkholder inequalities. In Section~5, we examine the case of sums of noncommuting independent sequences of mean zero in the sense of \cite{JX3}. In particular, we provide $\Phi$-moment analogues of the noncommutative Rosenthal inequalities from \cite{JX3}.  We also provide the corresponding Rosenthal inequalities for noncommuting   independent sequences in noncommutative symmetric spaces. In the last section,  we discuss possible future direction for general $\Phi$-moments and list some related open problems.
\section{Preliminaries}

\subsection{Orlicz functions and noncommutative Orlicz spaces}

Throughout this paper,  $\M$  will always denote  a semifinite von Neumann algebra equipped with a faithful normal semifinite trace $\T$.
Assume that $\M$ is acting on a Hilbert space $H$. A closed densely defined operator $x$ on $H$ is said to be  affiliated with $\M$ if $x$ commutes with every unitary $u$ in the commutant $\M'$ of $\M$. If $a$ is a densely defined self-adjoint operator on $H$ and $a=\int_{-\infty}^\infty s \ de_s^a$ is its spectral decomposition, then for any Borel subset $B\subseteq \mathbb{R}$, we denote by $\chi_B(a)$ the corresponding spectral projection $\int_{-\infty}^\infty \chi_B(s) \ de_s^a$. An operator $x$ affiliated with $\M$ is called $\T$-measurable if there exists $s>0$ such that
$\T(\chi_{(s,\infty)}(|x|))<\infty$. It is known that the set of all $\T$-measurable  operators with respect to $(\M,\T)$ is a topological $*$-algebra which we will denote by $L_0(\M,\T)$. We refer to \cite{N, PX3, TAK2}  for unexplained terminology.
 For $x \in L_0(\M,\T)$, define 
the distribution  function of $x$ by setting for $s>0$,
\[
\lambda_s(x)=\T\big( \chi_{(s, \infty)}(|x|)\big).
\]
The   generalized singular value of $x$ is defined by
\[
\mu_t(x)=\inf\{s>0; \lambda_s(x)\leq t \}, \quad t>0.
\]
The function $t \mapsto \mu_t(x)$ from $(0,\infty)$ into $[0,\infty)$ is right-continuous and nonincreasing (\cite{FK}).  
We note that 
 for the case where $\M$ is the abelian von Neumann algebra $L_\infty(0,\infty)$ with  the trace given by  integration  with respect to the Lebesgue  measure, $L_0(\M,\T)$  becomes  the linear space of all measurable functions $L_0(0,\infty)$ and $\mu(f)$ is the decreasing rearrangement of  the function $|f|$ in the sense of \cite{LT}.

By     \emph{an Orlicz function} $\Phi$ on $[0, \infty)$, we mean  a continuous,  increasing, and convex function such that $\Phi(0)=0$ and $\lim_{t \to \infty} \Phi(t)=\infty$.  For examples and basic properties of Orlicz functions we refer to \cite{Kras-Rutickii, Maligranda, Maligranda2}.

 Given   an operator  $x \in L_0(\M,\T)$ and an Orlicz function $\Phi$, we may define $\Phi(|x|)$ through functional calculus.  That is, if $|x|=\int_0^\infty s\  de_s^{|x|}$
 is the spectral decomposition of $|x|$, then
 \[
 \Phi(|x|)=\int_0^\infty \Phi(s) \ de_s^{|x|}.
 \]
 The operator $\Phi(|x|)$ is then  a positive $\T$-measurable operator. It is important  to observe that the trace of $\Phi(|x|)$ can be calculated using  either the distribution function  of $|x|$  or  the  singular value function  of $|x|$. Indeed,  one can easily deduce from \cite[Corollary~2.8]{FK} that  if $x \in L_0(\M, \T)$, then  we have the identities:
\[
\T\big[\Phi\big(|x|\big)\big] =\int_0^\infty \lambda_s\big(|x| \big)\ d\Phi(s)=\int_0^\infty \Phi\big( \mu_t(x)\big) \ dt.
\]
The quantity $\T\big[\Phi\big(|x|\big)\big]$ will be   referred to as the \emph{$\Phi$-moment} of the operator $|x|$. Clearly, if we consider  the power function $\Phi(t)=t^p$ for $1\leq p<\infty$, then this reduces to the usual  notion   of $p$-th moment of $|x|$. It is however important to point out  that in general $\Phi$-moments do not necessarily define  a norm and therefore many tools used for various results on $p$-th moments are no longer available when dealing with $\Phi$-moments.

We will assume throughout that $\Phi$ satisfies a growth condition known as \emph{the $\Delta_2$-condition}. That is, for some constant $C>0$,
\[
\Phi(2t) \leq C \Phi(t), \quad t\geq 0.
\]
It is easy to check that $\Phi$ satisfies the $\Delta_2$-condition if and only if  for every $a>0$, there exists a constant $C_a >0$  such that $\Phi(at) \leq C_a \Phi(t)$ for all $t>0$. More generally, by functional calculus, if $0\leq x \in  L_0(\M,\T)$ and $a$ is a positive scalar  then  the following operator inequality holds:
\[
\Phi(ax) \leq C_a \Phi(x).
\]
One  can also deduce from the integral representation stated above and \cite[Theorem~4.4(iii)]{FK}  that  if $(x_i)_{i=1}^n$  is a finite sequence in $L_0(\M)$ and $(\alpha_i)_{i=1}^n  \subset (0,1)^n$ with $\sum_{i=1}^n \alpha_i=1$ then 
\begin{equation}\label{convex}
\T\Big[ \Phi\big(\big| \sum_{i=1}^n \alpha_i x_i\big|\big)\Big] \leq   \sum_{i=1}^n\alpha_i\T\big[ \Phi\big(|x_i |\big)\big].
\end{equation}
As a consequence of \eqref{convex} and  the $\Delta_2$-condition, we have  the quasi-triangle inequality:
\[
\T\big[ \Phi\big(|x + y|\big)\big]\leq  C_\Phi\Big(\T\big[ \Phi\big(|x |\big)\big] +
\T\big[ \Phi\big(|y|\big)\big]\Big).
\]
These inequalities will be used repeatedly throughout. Next, we introduce some standard indices for Orlicz functions. For a given Orlicz function $\Phi$,  we let 
\[
M(t,\Phi)=\sup_{s>0} \frac{\Phi(ts)}{\Phi(s)}, \quad t>0,
\]
and
\[
p_\Phi=\lim_{t\to 0^+} \frac{\log\big( M(t,\Phi) \big)}{\log t}, \quad q_\Phi=\lim_{t\to \infty} \frac{\log\big( M(t,\Phi) \big)}{\log t}.
\]
These are known as \emph{Matuzewska-Orlicz indices} of the Orlicz function $\Phi$. For more information on these indices and their connections with other indices, we refer to the monographs  \cite{Maligranda, Maligranda2}. In general, $1\leq p_\Phi \leq q_\Phi \leq \infty$ and the $\Delta_2$-condition is equivalent to $q_\Phi<\infty$. 

We now recall  the definition of  Orlicz spaces. For a given  Orlicz function $\Phi$, 
the Orlicz  function space $L_\Phi(0,\infty)$  is the set of all Lebesgue measurable functions $f$ defined on $(0,\infty)$ such that for some  constant $c>0$,
\[
\int_0^\infty \Phi\Big({|f(t)|}/{c}\Big)\ dt <\infty.
\]
If we equip $L_\Phi(0,\infty)$ with the Luxemburg norm:
\[
\big\|f\big\|_{L_\Phi}=\inf\Big\{ c>0: \int_0^\infty \Phi\Big({|f(t)|}/{c}\Big)\ dt \leq 1\Big\},
\]
then $L_\Phi(0,\infty)$ is a fully symmetric Banach function space  in the sense of \cite{DDP4}.
Moreover, the Boyd indices of $L_\Phi(0,\infty)$ coincide with the indices $p_\Phi$ and $q_\Phi$ (see \cite{Maligranda}).
We may define 
the noncommutative Orlicz space $L_\Phi(\M,\T)$ following the general scheme  of constructing noncommutative analogue of symmetric function spaces as described in \cite{DDP1, DDP4, Kalton-Sukochev, X}. 
 Note that under the $\Delta_2$-condition, $x \in L_\Phi(\M,\T)$ if and only if $\T[\Phi(|x|)]<\infty$.  Also, it is clear that if  $\Phi(t)=t^p$  with $1\leq p<\infty$,  then $L_\Phi(\M,\T)=L_p(\M,\T)$ where $L_p(\M,\T)$ is the usual noncommutative $L_p$-space associated with $(\M,\T)$.

We now gather some  preliminary results on noncommutative Orlicz spaces that we will need in the sequel. We assume that the next lemma is known but we could not find any specific reference.  We feel  that  a proof  is needed since in general $\Phi$-moments do not define a norm.
\begin{lemma}\label{convergence}
 Let  $(x_n)_{n\geq 1}$ be  a sequence in $L_\Phi(\M,\T)$  and $x\in L_\Phi(\M,\T)$.
 \begin{enumerate}[{\rm(i)}]
\item If  $\lim_{n\to \infty} \|x_n-x\|_{L_\Phi(\M)}=0$ then
$\lim_{n\to\infty} \T\big[\Phi(|x_n|)\big]=\T\big[\Phi(|x|)\big]$.
\item If $(x_n)_{n\geq 1}$ converges to $x$ weakly in $L_\Phi(\M,\T)$ then
$\T\big[\Phi(|x|)\big] \leq \liminf_{n\to\infty} \T\big[\Phi(|x_n|)\big]$.
\end{enumerate}
\end{lemma}
\begin{proof} Let us begin with the first item. Recall that since $\Phi$ satisfies the $\Delta_2$-condition, a sequence $(f_n)_{n\geq 1}$ in $L_\Phi$ converges in norm to $f$ in $L_\Phi$ if and only if $\lim_{n\to \infty}\int_0^\infty \Phi(|f_n(t)-f(t)|)\ dt =0$. Therefore,  $\lim_{n\to \infty}\|x_n- x\|_{L_\Phi(\M)}=0$ if and only if  $\lim_{n\to \infty}\int_0^\infty\Phi( \mu_t(x_n-x))\ dt=0$.  We have from \cite[Theorem~3.4]{DDP1} that for every $n\geq 1$, the function  $|\mu(x_n)-\mu(x)|$ is submajorized by $\mu(x_n-x)$ in the sense that  for every $t>0$, 
\[
\int_0^t |\mu_s(x_n)-\mu_s(x)|\ ds  \leq  \int_0^t \mu_s(x_n - x) \ ds.
\]
Since $L_\Phi(0,\infty)$ is fully symmetric,  it follows that $\lim_{n\to \infty}\|\mu(x_n)-\mu(x)\|_{L_\Phi}=0$. next, we observe that 
$\{\Phi( \mu(x_n) ); n\geq 1\}$ is a uniformly integrable subset of $L_1(0,\infty)$.
This is the case since by  the $\Delta_2$-condition, there is a constant $C_\Phi$ so that for every $n\geq 1$, we have $\Phi(\mu(x_n)) \leq C_\Phi \Phi(|\mu(x_n)-\mu(x)|) + C_\Phi \Phi(\mu(x))$.

Now, fix an arbitrary  subsequence $(y_n)_{n\geq 1}$ of $(x_n)_{n\geq1}$.
 There exists a further  subsequence $(y_{n_k})_{k\geq 1}$   of $(y_n)_{n\geq 1}$ so that $\mu(y_{n_k}) \to \mu(x)$ a.e. By uniform integrability of $\{\Phi( \mu(x_n) ); n\geq 1\}$, we have
 \[\lim_{k\to \infty}\int_{0}^\infty \Phi(\mu_t(y_{n_k})) \ dt = \int_{0}^\infty \Phi(\mu_t(x))\  dt.\]
This is equivalent to $\lim_{k\to \infty} \T\big[ \Phi\big(|y_{n_k}|\big) \big]=\T\big[ \Phi\big(|x|\big) \big]$. Therefore, we have shown that every subsequence of 
 $\{\T\big[ \Phi\big(|x_n|\big) \big]\}_{n\geq 1}$ has a  further subsequence that converges to  $\T\big[ \Phi\big(|x|\big) \big]$. This proves  that
$
\lim_{n\to\infty} \T\big[ \Phi\big(|x_n|\big) \big]=\T\big[ \Phi\big(|x|\big) \big]
$
as claimed.

For the second item, assume that $x_n \to x$ weakly and  let $\xi$ be  a limit point of the bounded  sequence $\{\T\big[ \Phi\big(|x_n|\big) \big]\}_{n\geq1}$. Fix a subsequence $(y_n)$ of $(x_n)$ such that $\xi=\lim_{n\to \infty}\T\big[ \Phi\big(|y_n|\big) \big]$. Next, we choose a sequence $(z_n)$  consisting of block convex combinations of $(y_n)$ such that $\lim_{n\to \infty}\|z_n-x\|_{L_\Phi(\M)}=0$. From the first item, we have $\T\big[ \Phi\big(|x|\big) \big]=\lim_{n\to \infty}\T\big[ \Phi\big(|z_n|\big) \big]$. For each $n\geq 1$, write
$z_n=\sum_{j=p_n}^{q_n} \alpha_j y_j$ with $1\leq p_1<q_1<p_2<q_2 <\cdots$,  $\alpha_i \in [0,1]$ for all $i\geq 1$, and $\sum_{i=p_n}^{q_n} \alpha_i=1$ for all $n\geq 1$. It follows from \eqref{convex} that 
\begin{align*}
\T\big[ \Phi\big(|x|\big) \big] &=\lim_{n\to \infty}\T\big[ \Phi\big(|z_n|\big) \big]\\
&\leq  \lim_{n\to \infty} \sum_{i=p_n}^{q_n} \alpha_i \T\big[ \Phi\big(|y_i|\big) \big]\\
&=\lim_{n\to \infty} \T\big[ \Phi\big(|y_n|\big) \big]=\xi.
\end{align*}
The desired inequality follows from taking the infimum over all  such limit points.
\end{proof}

We now discuss  some background on complementary Orlicz functions. Let  $\Phi$ be an Orlicz function. It is well-known that $\Phi$ admits an integral representation
\[
\Phi(u)=\int_0^u \phi(s)\ ds,\quad u>0,
\]
where $\phi$ is a nondecreasing right-continuous function  defined on the interval $[0,\infty)$. The function $\phi$ is usually referred to as the right derivative of $\Phi$.
Let $\psi(t)=\sup\{s: \phi(s)\leq t\}$ be the right inverse of $\phi$. We observe  that $\psi$ is a nondecreasing right-continuous function on $[0,\infty)$ and  if $\phi$ is a continuous function then $\psi$ is the usual inverse of $\phi$. We define the  Orlicz \emph{complementary}  function to $\Phi$ by setting:
\[
\Phi^*(v)=\int_0^v \psi(t)\ dt, \quad v>0.
\]
Clearly, $\Phi^*$ is an Orlicz function and  under some natural conditions on $\Phi$, there is a  canonical  duality between the noncommutative  Orlicz spaces $L_\Phi(\M,\T)$ and $L_{\Phi^*}(\M,\T)$. 
 We refer to \cite[Chap.9]{Maligranda2} for more detailed  accounts of such duality in the commutative case. 
It is worth mentioning that  for the special case where  $\Phi(u)=u^r/r$ for some  $1<r<\infty$ then $\Phi^*(v)=  v^{r'}/{r'}$ where $r'$ denotes the index conjugate to $r$.  Therefore, we may view $\Phi^*$ as the Orlicz function analogue of the concept of index  conjugates. In fact,   from \cite[Corollary~11.6]{Maligranda2}, the indices of $\Phi^*$ satisfy:
\[
1/p_\Phi + 1/q_{\Phi^*} =1/p_{\Phi^*} +1/q_{\Phi} =1.
\]
 We refer to \cite[Chap.~I]{Kras-Rutickii} for more in depth discussion on   connections between $\Phi$  and $\Phi^*$.  Another fact that is of particular importance  for our purpose is the so-called \emph{Young's inequality} which states that for every $u,v \geq 0$, the following inequality holds:
\[
uv\leq \Phi(u) +\Phi^*(v).
\]

As an elementary application of Young's inequality, we record the following  lemma for further use.
\begin{lemma}\label{Young} For every $x \in L_\Phi(\M)$ and $y\in L_{\Phi^*}(\M)$, $xy \in L_1(\M)$ and
\[
\|xy\|_1 \leq \T\big[\Phi\big(|x|\big)\big] +  \T\big[\Phi^*\big(|y|\big)\big]. 
\]
\end{lemma}
\begin{proof} 
First, we note from basic properties of  generalized singular values that if   $xy \in L_1(\M)$ then
using properties of singular values (\cite[Theorem~4.2]{FK}),
\[
\|xy\|_1 =\int_0^\infty \mu_t(xy)\ dt
\leq  \int_0^\infty \mu_t(x)\mu_t(y)\ dt.
\]
By Young's inequality,  we deduce that
\[
\|xy\|_1\leq  \int_0^\infty \Phi(\mu_t(x))\ dt  +\int_0^\infty \Phi^*(\mu_t(y))\ dt,
\]
which is clearly the desired inequality.
\end{proof}

Our next result may be viewed as an operator reverse Young's inequality and could be of independent interest.

\begin{proposition}\label{duality}
Let $\Phi$ be an Orlicz function with $1<p_\Phi \leq q_\Phi <\infty$.  For every $0\leq x \in L_\Phi(\M)$ there exists $0\leq y \in L_{\Phi^*}(\M)$ such that $y$ commutes with $x$  and satisfies
\[
xy=\Phi(x) +\Phi^*(y).
\]
\end{proposition}
\begin{proof} We note first that since $p_\Phi>1$, we have $q_{\Phi^*}<\infty$ and therefore $\Phi^*$ satisfies the $\Delta_2$-condition. Let $\phi$ denote the right derivative of $\Phi$. The proposition is a consequence of the  following fact which can be found in \cite[p.13]{Kras-Rutickii} (see also \cite[ p.48]{Maligranda2}):
\[
uv= \Phi(u) + \Phi^*(v) \iff   v=\phi(u). 
\]
That is, at the function level, the following identity holds:
\[  
u\phi(u) =\Phi(u) +\Phi^*(\phi(u)), \quad u\geq 0.
\]
We remark that since the function $\phi$ is  monotone, it is Borel measurable. Using functional calculus on the positive  operator $x$,  the preceding identity yields:
\[
 x\phi(x) =\Phi(x) +\Phi^*(\phi(x)).
\]
It is enough to consider  $y=\phi(x)$. Clearly, $y\geq 0$ and commutes with $x$. To verify  that $y \in L_{\Phi^*}(\M)$, we appeal to another index
of $\Phi$ defined as follows: 
\[
b_\Phi :=\sup_{t>0} \frac{t\Phi'(t)}{\Phi(t)} =\sup_{t>0} \frac{t\phi(t)}{\Phi(t)}.
\]
In general, we only  have $q_\Phi \leq b_\Phi$  but the relevant property  we need is that   the $\Delta_2$-condition  is equivalent to $b_\Phi<\infty$.
These facts were taken from \cite[Theorem~3.2]{Maligranda}. The crucial observation  we make is that for every $t>0$,
\[
t\phi(t) \leq  b_\Phi \Phi(t), 
\]
where the function on the right hand side is finite for all $t>0$. Thus, 
by functional calculus and the definition of $y$, the preceding inequality yields the operator inequality:
\[
0\leq xy \leq b_\Phi \Phi(x).
\]
This  is equivalent to 
$
\Phi^*(y) \leq (b_\Phi -1) \Phi(x)$. Taking traces, we have
\[
\T\big[ \Phi^*(y)\big]   \leq (b_\Phi -1) \T\big[\Phi(x)\big].
\]
Since $x \in L_\Phi(\M)$,  the right hand side is finite and therefore, we have $\T\big[ \Phi^*(y)\big]<\infty$. As  $\Phi^*$ satisfies  the $\Delta_2$-condition, this  is equivalent to  $y \in L_{\Phi^*}(\M)$. The proof is complete.
\end{proof}


\subsection{Noncommutative martingales}
Let us now review the general setup for noncommutative martingales. For simplicity, we assume for the remaining of the paper that $\M_*$ is separable.
In the sequel, we always denote by $(\M_n)_{n \geq 1}$ an
increasing sequence of von Neumann subalgebras of ${\M}$
whose union  is weak*-dense in
$\M$. For $n\geq 1$, we assume that there exists a trace preserving conditional expectation ${\E}_n$ 
from ${\M}$ onto  ${\M}_n$.  It is well-known that  for $1\leq p\leq \infty$, $\E_n$ extends to a contractive projection from $L_p(\M,\T)$ onto $L_p(\M_n, \T_n)$,  where $\T_n$  denotes the restriction of $\T$ on $\M_n$. More generally, if $\Phi$ is an Orlicz function, then  since $L_\Phi(0,\infty)$ is  fully symmetric,   it follows that  $\E_n$  is a contractive  projection  from $L_\Phi(\M,\T)$ onto $L_\Phi(\M_n,\T_n)$ (see for instance, \cite[Proposition~2.1]{Dirk-Pag-Pot-Suk}). 

\begin{definition}
A sequence $x = (x_n)_{n\geq 1}$ in $L_1(\M)+\M$ is called \emph{a
noncommutative martingale} with respect to $({\M}_n)_{n \geq
1}$ if $\mathcal{E}_n (x_{n+1}) = x_n$ for every $n \geq 1.$
\end{definition}
If in addition, all $x_n$'s belong to $L_\Phi(\M)$  for a given Orlicz function $\Phi$, then  $x$ is called an $L_\Phi(\M)$-martingale.
In this case, we may define
\begin{equation*}\| x \|_{L_\Phi(\M)} = \sup_{n \geq 1} \|
x_n \|_{L_\Phi(\M)}.
\end{equation*}
For the case where   $\| x \|_{L_\Phi(\M)} < \infty$, then $x$ is called
a bounded $L_\Phi(\M)$-martingale. We note that if  the indices of $\Phi$ satisfy $1<p_\Phi \leq q_\Phi<\infty$, then $L_\Phi(\M)$ is a reflexive space.  In this case,  any  bounded $L_\Phi(\M)$-martingale $(x_n)_{n\geq 1}$  converges to some $x_\infty$ in $L_\Phi(\M)$ that satisfies $\E_n(x_\infty)=x_n$ for all $n\geq 1$. From this fact, whenever $1<p_\Phi\leq q_\Phi<\infty$, we will not make any distinction between  operators in $L_\Phi(\M)$ and bounded $L_\Phi(\M)$-martingales.

Let $x = (x_n)_{n\ge1}$ be a noncommutative martingale with respect to
$(\M_n)_{n \geq 1}$.  Define $dx_n = x_n - x_{n-1}$ for $n
\geq 1$ with the usual convention that $x_0 =0$. The sequence $dx =
(dx_n)_{n\ge1}$ is called the \emph{ martingale difference sequence} of $x$. 

In this paper, we will be mainly working with conditioned square functions and noncommutative conditioned Hardy spaces. We refer the reader to \cite{Bekjan-Chen, RW} for noncommutative Hardy spaces associated with  square functions. Recall  that if $x=(x_n)_{n\geq 1}$ is an $L_2(\M)+\M$-martingale,   then we can formally define:
\begin{equation}\label{cond-square-original}
 s_c (x) = \Big ( \sum_{k \geq 1} \E_{k-1}|dx_k |^2 \Big )^{1/2}
 \
\text{and}
 \ \ 
 s_r (x) = \Big ( \sum_{k \geq1} \E_{k-1}| dx^*_k |^2 \Big)^{1/2}.
 \end{equation}
These are called the column and row \emph{conditioned square functions of $x$}, respectively. We want  to emphasize that when $dx_k \notin L_2(\M) +\M$, then $|dx_k|^2$ may not be necessary in $L_1(\M) +\M$. Therefore, $\E_{k-1}|dx_k|^2$ is not necessarily a well-defined object. Thus,   extra cares are needed for martingales that do not belong to $L_2(\M) +\M$. Since the main topic of this paper is dealing with 
various inequalities involving  conditioned square functions, we will review the general construction which is based on  the so-called  conditioned spaces. These  were formally introduced  by  Junge  in \cite{Ju} for noncommutative $L_p$-spaces and were extensively used by Junge and Xu in \cite{JX,JX3}. Recently, these ideas were adapted in \cite{RW} to the case of  more general classes of noncommutative symmetric spaces. 
Below, we  use the usual  convention that $\E_0=\E_1$.

Let $\E:\M \to \N$ be a normal faithful conditional expectation, where $\N$ is a von Neumann subalgebra of $\M$. 
For $0<p\leq \infty$, we define  the conditioned space    $L_p^c(\M,\E)$ to be the completion of  $\M \cap L_p(\M)$ with respect to the quasi-norm
\[
\big\|x\big\|_{L_p^c(\M,\E)} =\big\|\E(x^*x)\big\|_{p/2}^{1/2}.
\]
It was shown in \cite{Ju} that for every $n$ and $0<p \leq \infty$,  there exists an isometric right $\M_n$-module map $u_{n,p}: L_p^c(\M, \E_n) \to L_p(\M_n;\ell_2^c)$ such that if $(e_{i,j})_{i,j \geq 1}$ is the family of unit matrices in $B(\ell_2(\mathbb{N}))$, then
\begin{equation}\label{u}
u_{n,p}(x)^* u_{n,q}(y)=\E_n(x^*y) \otimes e_{1,1},
\end{equation}
for all $x\in L_p^c(\M;\E_n)$ and $y \in L_q^c(\M;\E_n)$ with $1/p +1/q \leq 1$.
We now consider  the increasing sequence of expectations $(\E_n)_{n\geq 1}$.  Denote by $\mathcal{F}$ the collection of all finite sequences $(a_n)_{n\geq 1}$ in $L_1(\M) \cap \M$.
 For $0< p\leq \infty$,  define  the space $L_p^{\rm cond}(\M; \ell_2^c)$ to be   the completion  of $\mathcal{F}$  with respect to the (quasi) norm:
\begin{equation}\label{conditioned-norm}
\big\| (a_n) \big\|_{L_p^{\rm cond}(\M; \ell_2^c)} = \big\| \big(\sum_{n\geq 1} \E_{n-1}|a_n|^2\big)^{1/2}\big\|_p.
\end{equation}
The space $L_p^{\rm cond}(\M;\ell_2^c)$ can be isometrically embedded into an $L_p$-space associated to  a semifinite von  Neumann algebra  by means of the following map: 
\[
U_p : L_p^{\rm cond}(\M; \ell_2^c) \to L_p(\M \overline{\otimes} B(\ell_2(\mathbb{N}^2)))\]
defined by setting 
\[
U_p((a_n)_{n\geq 1}) = \sum_{n\geq 1} u_{n-1,p}(a_n) \otimes e_{n,1}
\]
 From \eqref{u}, it follows  
   that if $(a_n)_{n\geq 1} \in L_p^{\rm cond}(\M;  \ell_2^c)$ and $(b_n)_{n\geq 1} \in L_q^{\rm cond}(\M;  \ell_2^c)$  for $1/p +1/q \leq 1$ then
\begin{equation}\label{c-square}
U_p( (a_n))^* U_q((b_n)) =\big(\sum_{n\geq 1} \E_{n-1}(a_n^* b_n)\big) \otimes e_{1,1} \otimes e_{1,1}.
\end{equation}
In particular,  $\| (a_n) \|_{L_p^{\rm cond}(\M; \ell_2^c)}=\|U_p( (a_n))\|_p$ and 
hence  $U_p$ is indeed an isometry.  We note that $U_p$ is independent of $p$ in the sense of interpolation. Below, we will simply write $U$ for $U_p$.
We refer the reader to \cite{Ju} and \cite{Junge-Perrin} for more details on the preceding construction.

 In  \cite{RW},   the notion of conditioned spaces  were generalized  to the general context of noncommutative symmetric spaces.  We will only need here the special case of  noncommutative Orlicz spaces. We include the details for further use.
 
  We consider  the  algebraic linear  map $U$ restricted to   the linear space $\mathcal{F}$  that takes  its values in $L_1(\M \overline{\otimes} B(\ell_2(\mathbb{N}^2))) \cap \M \overline{\otimes} B(\ell_2(\mathbb{N}^2))$.  For a given 
   sequence $(a_n)_{n\geq 1} \in \mathcal{F}$, we set:
\[
\big\| (a_n) \big\|_{L_\Phi^{\rm cond}(\M; \ell_2^c)} = \big\| \big(\sum_{n\geq 1} \E_{n-1}|a_n|^2\big)^{1/2}\big\|_{L_\Phi(\M)}=\big\| U( (a_n))\big\|_{L_\Phi(\M \overline{\otimes} B(\ell_2(\mathbb{N}^2)))}.
\]
This is well-defined and  induces   a norm  on the linear space $\mathcal{F}$.  We define   the Banach space $L_\Phi^{\rm cond}(\M;  \ell_2^c)$ to be the completion of $\mathcal{F}$  with respect to the above  norm. Then  
 $U$ extends to an isometry from  $L_\Phi^{\rm cond}(\M;\ell_2^c)$ into $L_\Phi(\M \overline{\otimes} B(\ell_2(\mathbb{N}^2)))$ which we will still denote by $U$.

Similarly, we may define the corresponding  row version $L_\Phi^{\rm cond}(\M; \ell_2^r)$ which can also be viewed as a subspace of $L_\Phi(\M \overline{\otimes} B(\ell_2(\mathbb{N}^2)))$ as row vectors. 

Now we define the column/row  conditioned Orlicz-Hardy spaces. Let $\mathcal{F}_M$ denote  the set of all finite martingales in $L_1(\M) \cap \M$.
Define  $\h_\Phi^c (\mathcal{M})$ (respectively,  $\h_\Phi^r (\mathcal{M})$) as the completion  of $\mathcal{F}_M$  under the  norm $\| x \|_{\h_\Phi^c}=\| s_c (x) \|_{L_\Phi(\M)}$
(respectively,  $\| x \|_{\h_\Phi^r}=\| s_r (x) \|_{L_\Phi(\M)} $). We observe that  for every $x \in \mathcal{F}_M$, $\|x\|_{\h_\Phi^c} =
\| (dx_n)\|_{L_\Phi^{\rm cond}(\M; \ell_2^c)}$. Therefore, $\h_\Phi^c(\M)$ may be viewed as a subspace of $L_\Phi^{\rm cond}(\M; \ell_2^c)$.  More precisely, we consider the map
$\mathcal{D}: \mathcal{F}_M \to \mathcal{F}$ by setting $\mathcal{D}(x)= (dx_n)_{n\geq 1}$. Then $\mathcal{D}$ extends  to an isometry from $\h_\Phi^c(\M)$ into $L_\Phi^{\rm cond}(\M;\ell_2^c)$ which we will  denote by $\mathcal{D}_c$. In the sequel, we will make frequent use of the isometric embedding:
\[
U\mathcal{D}_c: \h_\Phi^c(\M)  \to  
L_\Phi(\M \overline{\otimes} B(\ell_2(\mathbb{N}^2))).
\]
We can make  similar  assertions for the row case. That is, $\h_\Phi^r(\M)$ embeds isometrically into $L_\Phi(\M \overline{\otimes} B(\ell_2(\mathbb{N}^2)))$.
We also need the diagonal Hardy space  $\h_\Phi^d(\M)$  which is the space of all martingales whose martingale difference sequences belong to $L_\Phi(\M \overline{\otimes} \ell_\infty)$ equipped with the norm $\|x\|_{\h_\Phi^d} := \|(dx_n)\|_{L_\Phi(\M \overline{\otimes} \ell_\infty)}$.  As above, we denote by $\mathcal{D}_d$ the isometric extension of $\mathcal{D}$ from $\h_\Phi^d(\M)$ into $L_\Phi(\M \overline{\otimes} \ell_\infty)$.  From boundedness of conditional expectations, one can easily verify that  $\mathcal{D}_d(\h_\Phi^d(\M))$ is a closed subspace of $L_\Phi(\M \overline{\otimes} \ell_\infty)$ which implies in turn  that $\h_\Phi^d(\M)$ is a Banach space.
 As noted in \cite{RW}, $\h_\Phi^d(\M)$, $\h_\Phi^c(\M)$, and $\h_\Phi^r(\M)$ are compatible in the sense that they embed  into a  larger Banach space. 
We now define the conditioned version of martingale Orlicz-Hardy spaces as follows.
If $1\leq p_\Phi \leq q_\Phi<2$, then
\begin{equation*}
\h_\Phi(\mathcal{M}) = \h_\Phi^d (\mathcal{M}) +  \h_\Phi^c (\mathcal{M})
+  \h_\Phi^r (\mathcal{M})
\end{equation*}
equipped with the  norm
\begin{equation*}
\| x \|_{\h_\Phi} = \inf \big\{ \| w \|_{ \h_\Phi^d} + \| y
\|_{ \h_\Phi^c} + \| z \|_{ \h_\Phi^r} \big\},
\end{equation*}
where the infimum is taken over all $w \in  \h_\Phi^d
(\M)$,  $y \in  \h_\Phi^c (\M)$, and $ z \in \h_\Phi^r
(\M)$ such that $ x = w + y + z.$
 If $2 \leq p_\Phi \leq q_\Phi <\infty$, then
\begin{equation*}\h_\Phi (\M) =  \h_\Phi^d
(\M) \cap  \h_\Phi^c (\M) \cap  \h_\Phi^r (\M)
\end{equation*}equipped with the norm\begin{equation*}
\| x \|_{\h_\Phi} = \max \big\{ \| x \|_{ \h_\Phi^d}, \|x
\|_{ \h_\Phi^c}, \| x \|_{ \h_\Phi^r} \big\}.
\end{equation*}
The  reason behind the consideration of different definitions according to $q_\Phi<2$
or $p_\Phi>2$ goes back to the  noncommutative Khintchine inequalities from \cite{LP4,LPI}.
For the particular case $\Phi(t)=t^p$ then $\h_\Phi(\M)=\h_p(\M)$  where $\h_p(\M)$ is the conditioned Hardy space as defined in \cite{Ju, JX}.  The space $\h_\Phi(\M)$ is the conditioned version of  martingale Orlicz Hardy spaces constructed from square  functions explicitly defined in \cite{Bekjan-Chen}.  As  a particular case of   the extensions of  the noncommutative  Burkholder inequalities to general  noncommutative symmetric spaces treated in \cite[Theorem~3.1]{RW}, we have 
the following identification:
\begin{equation}
\h_\Phi(\M) \approx_\Phi L_\Phi(\M)
\end{equation}
whenever $1<p_\Phi \leq q_\Phi<2$ or $2<p_\Phi\leq q_\Phi<\infty$.

\medskip

Let us now discuss $\Phi$-moments of conditioned square functions  for $x \notin L_2(\M) +\M$.   The important fact revealed by \eqref{c-square} is that if  $x$ is a martingale from  $\mathcal{F}_M$  then $s_c(x)$ (as defined above) can be identified  to the modulus of  the measurable operator $ U\mathcal{D}_c(x)$ in the space  $L_\Phi(\M \overline{\otimes} B(\ell_2(\mathbb{N}^2)))$. We extend this identity to  all  martingales  $x \in \h_\Phi^c(\M)$.  That is,  for each $x \in \h_\Phi^c(\M)$, we make  the convention  that  the column conditioned square function of  $x$  is given by:
\begin{equation}\label{cond-square}
s_c(x)= | U\mathcal{D}_c(x)|.
\end{equation}
Similarly, we may also define the corresponding  row version  by setting:
\[
s_r(y)=s_c(y^*)=| U\mathcal{D}_c(y^*)|,  \quad y \in \h_\Phi^r(\M).
\]
Clearly, the definition of $\h_\Phi^c(\M)$ allows  the identification for the norms:
\[
\big\|s_c(x)\big\|_{L_\Phi(\M)}= \big\| U\mathcal{D}_c(x) \big\|_{L_\Phi(\M \overline{\otimes}B(\ell_2(\mathbb{N}^2)))} =\big\| x \big\|_{\h_\Phi^c}.
\]
Accordingly,  $\Phi$-moments of  column conditioned square functions are  then understood as:
\begin{equation}\label{convension-moment}
\T\big[\Phi(s_c(x)) \big]= \T \otimes \Tr\big[\Phi\big(| U\mathcal{D}_c(x)|\big)\big],
\end{equation}
where $\Tr$ denotes the usual trace on $B(\ell_2(\mathbb{N}^2))$. We should warn the reader that when $x \notin L_2(\M) +\M$,  the $\Phi$-moment  $\T\big[\Phi(s_c(x)) \big]$ is only a suggestive notation as $s_c(x)$ may not exist in the sense of \eqref{cond-square-original}.  We also define $\T\big[\Phi(s_r(x)) \big]$ in a similar way.

 We end this subsection by recording the following  simultaneous decomposition result  that  we will  need in the sequel. 

\begin{theorem}[\cite{RW}]\label{simultaneous} There exists a family  $\{\kappa_p : 1<p<2\} \subset \mathbb{R}_+$  satisfying the following:
 if 
 $x\in L_1(\M) \cap L_2(\M)$, then there exist   $a \in \bigcap_{1<p<2} \h_p^d(\M)$, $b \in \bigcap_{1<p<2} \h_p^c(\M)$, and $c \in \bigcap_{1<p<2} \h_p^r(\M)$  such that:
\begin{enumerate}[{\rm(i)}]
\item  $x=a +b+ c$;
\item for every $1<p< 2$, the following inequality holds:
\[
\big\| a \big\|_{\h^d_p}
 + \big\| b \big\|_{\h_p^c} + \big\| c
\big\|_{\h_p^r}\leq \kappa_p \big\| x \big\|_p.
\]
\end{enumerate}
\end{theorem}
\section{Interpolations and some key inequalities}

In this section, we recall some basic definitions  from interpolation theory and  provide  four inequalities that are at the core of our argument in the next section. These are stated in  Proposition~\ref{h-interpolation}, Proposition~\ref{K-iteration}, Proposition~\ref{J-iteration}, and Proposition~\ref{KJ-comparison}. Although we only need these results in the special case  of various noncommutative $L_p$-spaces, for the sake of clarity, we chose to work with the abstract context of compatible couple of general Banach spaces. Our main references for interpolation of general Banach spaces  are   \cite{BENSHA, BL, KaltonSMS}.

 Let $\overline{X}=(X_0, X_1)$ be a compatible couple of Banach spaces in the sense that $X_0$ and $X_1$ are continuously embedded into a Hausdorff topological vector space. Then we can form the sum $\Sigma(\overline{X})=X_0 +X_1$ and the intersection $\Delta(\overline{X})=X_0 \cap X_1$ which are Banach spaces under the norms
 \[
 \big\|x\big\|_{\Sigma(\overline{X})} =\inf\Big\{ \big\|x_0\big\|_{X_0} + \big\|x_1 \big\|_{X_1} : x=x_0 + x_1,  x_0 \in X_0, x_1 \in X_1 \Big\}
 \]
and
\[
\big\|x\big\|_{\Delta(\overline{X})} = \max\Big\{\|x\|_{X_0}, \|x\|_{X_1} \Big\},
\]
respectively. A Banach space $Z$ will be called an \emph{intermediate space} with respect to $\overline{X}$ if 
$\Delta(\overline{X}) \subseteq Z \subseteq \Sigma(\overline{X})$ with continuous embeddings. An intermediate space $Z$ is called an \emph{interpolation space} if whenever  a bounded linear operator $T: \Sigma(\overline{X}) \to \Sigma(\overline{X})$ is such that $T(X_0) \subseteq X_0$ and $T(X_1) \subseteq X_1$, we have $T(Z)\subseteq Z$
and \[
\|T:Z\to Z\|\leq C\max\Big\{\|T : X_0 \to X_0\| , \|T:X_1 \to X_1\|\Big\}\]
 for some constant $C$.  In this case, we write $Z \in {\rm Int}(X_0, X_1)$.
Examples  of interpolation spaces that are  relevant to  this article are  Orlicz spaces. Indeed,  we have $L_\Phi  \in {\rm Int}(L_{p_0}, L_{p_1})$ whenever $p_0<p_\Phi \leq q_\Phi <p_1$. In fact, 
the following noncommutative generalization of the classical  Marcinkiewicz interpolation of operators  was  used  in \cite{Bekjan-Chen} as one of the main tools  for dealing with various $\Phi$-moment inequalities.
 We only state here the version we need.
\begin{theorem}[{\cite[Theorem~2.1]{Bekjan-Chen}}]\label{Bekjan-Chen}
Let $\M_1$ and $\M_2$ be  two semifinite von Neumann algebras equipped with normal semifinite faithful traces  $\T_1$ and $\T_2$, respectively.  Assume that $1\leq p_0<p_1 \leq \infty$.  Let $T: L_{p_0}(\M_1) +L_{p_1}(\M_1) \to L_{p_0}(\M_2) +L_{p_1}(\M_2)$ be a linear operator that satisfies $T(L_{p_i}(\M_1)) \subseteq  L_{p_i}(\M_2)$ for 
$i=0,1$. If $\Phi$ is an Orlicz function with $p_0<p_\Phi \leq q_\Phi <p_1$, then there exists a constant $C$ depending only on $p_0$, $p_1$, and $\Phi$ such that for every $x \in L_\Phi(\M_1)$,
\[
\T_2\big[\Phi\big(|Tx|\big)\big] \leq C \T_1\big[\Phi\big(|x|\big)\big].
\]
\end{theorem}

The following properties of  conditioned  Orlicz-Hardy spaces  and diagonal Orlicz Hardy spaces are taken from \cite[Proposition 2.8]{RW}.
\begin{lemma}\label{comp-Int} Assume that $1<p_0  <p_\Phi \leq q_\Phi < p_1<\infty$.  Then:
\begin{enumerate}[{\rm(i)}]
\item $\h_\Phi^d(\M)$ is complemented in $L_\Phi(\M \overline{\otimes}\ell_\infty)$;
\item $\h_\Phi^c(\M)$ is complemented in $L_\Phi(\M \overline{\otimes} B(\ell_2(\mathbb{N}^2)))$;
\item for $s\in\{d,c,r\}$, we have $\h_\Phi^s(\M) \in {\rm Int}(\h_{p_0}^s(\M), \h_{p_1}^s(\M))$.
\end{enumerate}
\end{lemma}
The next proposition is the Hardy space versions of Theorem~\ref{Bekjan-Chen}.
\begin{proposition}\label{h-interpolation}  Let $\N$ be a semifinite von Neumann algebra equipped with a  normal semifinite faithful  trace $\sigma$.  Assume that $1< p_0 <p_1 <\infty$. Let  $s \in \{d,c\}$ and $T:  \h_{p_0}^s(\M) + \h_{p_1}^s(\M) \to L_{p_0}(\N) + L_{p_1}(\N)$ be a linear operator that satisfies $T(\h_{p_i}^s(\M)) \subseteq  L_{p_i}(\N)$ for 
$i=0,1$. If $\Phi$ is an Orlicz function with $p_0<p_\Phi \leq q_\Phi <p_1$, then there exists a constant $C$ depending only on $p_0$, $p_1$, and $\Phi$ such that:
\begin{enumerate}[{\rm(i)}]
\item If $s=d$ and $x \in \h_\Phi^d(\M)$, then
$
\sigma\big[ \Phi\big(|Tx|\big)\big] \leq C \sum_{n\geq 1}  \T\big[\Phi\big(|dx_n|\big)\big]$.
\item If $s=c$ and $y \in \h_\Phi^c(\M)$, then 
$
\sigma\big[ \Phi\big( |Ty| \big)\big] \leq C \T\big[ \Phi\big(s_c(y)\big)\big]$.
\end{enumerate}
Similarly, if $S: L_{p_0}(\N) + L_{p_1}(\N) \to \h_{p_0}^s(\M) + \h_{p_1}^s(\M)$ is 
a linear operator that satisfies
 $S(L_{p_i}(\N)) \subseteq  \h_{p_i}^s(\M)$ for $i=0,1$, then there exists a constant $C$ depending only on  $p_0$, $p_1$, and $\Phi$ such that:
 \begin{enumerate}
 \item[{\rm(iii)}] If $s=d$ and $x \in L_\Phi(\N)$, then
 $
 \sum_{n\geq 1} \T\big[  \Phi\big( | d_n(Sx)| \big) \big] \leq C\sigma\big[ \Phi\big(|x|\big) \big]$ where  $(d_n(Sx))_{n\geq 1}$ denotes the martingale difference sequence of the martingale associated with $Sx$.
 \item[{\rm(iv)}] If $s=c$ and $y \in L_\Phi(\N)$, then 
 $
 \T\big[\Phi\big( s_c(Sy)\big) \big] \leq C \sigma\big[ \Phi\big(|y|\big) \big]$.
 \end{enumerate}
\end{proposition}
\begin{proof}
We begin with the diagonal part. Let
 $\Theta: L_{p_0}(\M \overline{\otimes}\ell_\infty) + L_{p_1}(\M \overline{\otimes}\ell_\infty)  \to \h_{p_0}^d(\M) + \h_{p_1}^d(\M)$ be the bounded projection  defined by:
$\Theta\big((a_n)_{n\geq 1}\big) = \sum_{n\geq 1} \E_{n}(a_n) -\E_{n-1}(a_n)$.
 It is clear  that   $T\Theta[L_{p_i}(\M \overline{\otimes}\ell_\infty) ] \subset L_{p_i}(\N)$ for $i=0,1$.  It follows from  Theorem~\ref{Bekjan-Chen} that   $T\Theta[L_\Phi(\M \overline{\otimes}\ell_\infty)] \subset L_\Phi(\M)$ and there exists a constant $C=C(p_0,p_1,\Phi)$ such that:
 \[
 \sigma\big[ \Phi\big( | T\Theta((a_n)_n)| \big)\big] \leq C \T \otimes \gamma\big[\Phi\big(| (a_n)_n|\big)\big]
 \]
where $\T \otimes \gamma$ is the natural trace of  $\M \overline{\otimes} \ell_\infty$.
Let $x \in \h_\Phi^d(\M)$. 
When applied to the operator $\mathcal{D}_d(x)\in L_\Phi(\M \overline{\otimes} \ell_\infty)$, the  above inequality yields the desired inequality.

Now, we verify the column case.  Let $\mathcal{S}=\M \overline{\otimes} B(\ell_2(\mathbb{N}^2))$ equipped with its natural trace $\T \otimes \Tr$. Define $\varPi: L_{p_0}(\mathcal{S}) + L_{p_1}(\mathcal{S}) \to \h_{p_0}^c(\M) +\h_{p_1}^c(\M)$ be the projection guaranteed by Lemma~\ref{comp-Int}. Then  we have, $T\varPi[L_{p_i}(\mathcal{S}) ] \subset L_{p_i}(\N)$ for $i=0,1$.
  As above,  we deduce from Theorem~\ref{Bekjan-Chen}  that  $T\varPi[L_\Phi(\mathcal{S})] \subset L_\Phi(\N)$ and there exists a constant $C=C_{ p_0,p_1,\Phi}$ such that for every $a \in L_\Phi(\mathcal{S})$,
 \[
 \sigma\big[ \Phi\big( | T\varPi(a)| \big)\big] \leq C \T \otimes \Tr\big[\Phi\big(| a|\big)\big].
 \]
Let  $y \in \h_\Phi^c(\M)$ and take $a=U\mathcal{D}_c(y)$. For this special case,   the preceding inequality clearly  translates into the inequality in item~{\rm{(ii)}}.

Items {\rm{(iii)}} and {\rm(iv)} follow from composing $S$ with the isometric embeddings   $\mathcal{D}_d: \h_{p_i}^d(\M) \to L_{p_i}(\M \overline{\otimes} \ell_\infty)$ and  $U\mathcal{D}_c: \h_{p_i}^c(\M) \to L_{p_i}(\mathcal{S})$ for $i=0,1$.
\end{proof}

We now turn our attention to specific types of interpolations.
A fundamental notion for real interpolation theory is the $K$-functional. This is given by setting:
\[
K(t,x)=K(t, x;\overline{X})=\inf\Big\{ \big\|x_0\big\|_{X_0} + t\big\|x_1\big\|_{X_1} : x=x_0 +x_1\Big\}, \ \ x \in \Sigma(\overline{X}).
\]
We will also need  a dual notion known as the $J$-functional defined by
\[
J(t,x)=J(t, x;\overline{X})=\max\Big\{ \big\|x\big\|_{X_0},  t\big\|x\big\|_{X_1} \Big\}, \ \ x \in \Delta(\overline{X}).
\]
These two notions will be heavily used in the sequel. 

We recall that by a representation of  $x \in \Sigma(\overline{X})$ with respect to the couple $\overline{X}$,  we mean  a measurable function $u: (0,\infty) \to \Delta(\overline{X})$ satisfying
\[
x=\int_{0}^\infty u(t)\ \frac{dt}{t}
\]
where the convergence of the integral is taken  in $\Sigma(\overline{X})$. Similarly, a discrete representation of $x$  with respect to the couple $\overline{X}$ is a series 
\[
x=\sum_{\nu \in \mathbb{Z}} u_\nu
\]
 with 
$u_\nu \in \Delta(\overline{X})$  for all $\nu \in \mathbb{Z}$  and the convergence of the series taken in the Banach space  $\Sigma(\overline{X})$.

 \begin{definition}  Given a compatible  couple $\overline{X}$ and $0\leq \theta \leq 1$, we say that an intermediate space $Z$  of $\overline{X}$ belongs to 
 
 (i) the class $\mathcal{C}_K(\theta, \overline{X})$  if  there exists a constant $C_1$ such that  for every $x \in Z$ and $t>0$, the following holds:
 \[
 K(t, x) \leq C_1 t^\theta \|x\|_Z.
\]

(ii)  the class $\mathcal{C}_J(\theta, \overline{X})$  if  there exists a constant $C_2$ such that  for every $x \in \Delta(\overline{X})$  and $t>0$, the following holds:
 \[
 \|x\|_Z \leq C_2 t^{-\theta}J(t, x). 
\]
 \end{definition}

Examples of spaces belonging to the class $\cal{C}_K(\theta, \overline{X})$ are those real interpolation spaces constructed using the $K$-method. Namely, the spaces
$(X_0, X_1)_{\theta,p, K}$ (we refer to \cite{BL} for the definition of $\|\cdot\|_{\theta,p, K}$). The corresponding  statement  is also valid for the class $\cal{C}_J(\theta, \overline{X})$. That is, $(X_0, X_1)_{\theta, p,J}$ belongs to $\cal{C}_J(\theta, \overline{X})$. In particular, for $\theta=1-p^{-1}$, $L_p$  belongs to both $\cal{C}_K(\theta)$ and $\cal{C}_J(\theta)$ for the couple $(L_1, L_\infty)$. A noncommutative analogue of the latter statement will be used in the sequel.
 
The next two propositions deal with reiteration type inequalities involving convex functions. 
 
\begin{proposition}\label{K-iteration} Let  $\overline{X}=(X_0,X_1)$ and $\overline{Y}=(Y_0, Y_1)$  be compatible couples  of Banach spaces and $0\leq \theta_0 <\theta_1\leq 1$. Assume that
$Y_i$ belongs to the class 
$\mathcal{C}_K(\theta_i, \overline{X})$ for $i=0,1$.   Then  the following inequality holds:
\[
\int_0^\infty  \Phi\big[ t^{-1} K(t, y; \overline{X}) \big]\ dt \lesssim_{\Phi, \theta_0,\theta_1}
\int_0^\infty \Phi\big[ t^{-1 +\theta_0} K( t^{\theta_1-\theta_0}, y ; \overline{Y}) \big]\ dt,
\quad y \in \Sigma(\overline{Y}).
\]
\end{proposition}
\begin{proof} From the assumptions,
there exist constants $C_0$ and $C_1$  such that if $y= y_0 + y_1  \in \Sigma(\overline{Y})$ then 
for every $t>0$,
\[
K(t, y_0; \overline{X}) \leq  C_0 t^{\theta_0} \|y_0\|_{Y_0} \ \text{and}\ 
K(t, y_1; \overline{X})  \leq  C_1 t^{\theta_1} \|y_1\|_{Y_1}.
\] 
It follows that  
$
K(t,y; \overline{X}) \leq C_0 t^{\theta_0} \|y_0\|_{Y_0}  + C_1 t^{\theta_1} \|y_1\|_{Y_1}$.
Taking the  infimum over  all such decompositions of $y$, we have for $C=\max\{C_0, C_1\}$ that 
\[
K(t,y; \overline{X} ) \leq C t^{\theta_0} K(t^{\theta_1-\theta_0}, y; \overline{Y}).
\]
Since $\Phi$ is increasing and satisfies the $\Delta_2$-condition, we may conclude that
\begin{align*}
\int_0^\infty \Phi\big[ t^{-1}K(t,y; \overline{X}) \big]  \ dt &\leq  
  \int_0^\infty \Phi\big[Ct^{-1 +\theta_0} K(t^{\theta_1-\theta_0}, y; \overline{Y})\big] \  dt \\
  &\lesssim
 \int_0^\infty \Phi\big[t^{-1 +\theta_0} K(t^{\theta_1-\theta_0}, y; \overline{Y})\big] \  dt.  
\end{align*}
The fact that the constant depends only on $\Phi$, $\theta_0$, and $\theta_1$ is clear from  the argument.
\end{proof}
A dual version of the preceding  proposition reads as follows:
\begin{proposition}\label{J-iteration} Let  $\overline{X}=(X_0,X_1)$ and $\overline{Y}=(Y_0, Y_1)$  be  compatible couples  of Banach spaces and $0\leq \theta_0 <\theta_1\leq 1$. Assume that
$Y_i$ belongs to the class 
$\mathcal{C}_J(\theta_i, \overline{X})$ for $i=0,1$.  Let $y \in \Delta(\overline{Y})$  and 
 $u(\cdot)$ be a representation of $y$  for the  couple $\overline{X}$.  If  
 $u(\cdot)$ is  also a representation of $x$ for the  couple $\overline{Y}$ then  the following inequality holds:
\[
\int_0^\infty \Phi\big[ t^{-1 +\theta_0} J(t^{\theta_1-\theta_0}, u(t); \overline{Y}) \big] \ dt \lesssim_{\Phi, \theta_0,\theta_1}
\int_0^\infty \Phi\big[ t^{-1}J(t, u(t); \overline{X})\big]\ dt.
\]
\end{proposition}

\begin{proof} The argument is nearly identical to the one used earlier. We include the details for completeness.  For the inequality,  we have from the assumptions that there exist constants $C_0$ and $C_1$ such that for every $t>0$,
\[
\|u(t)\|_{Y_0} \leq  C_0 t^{-\theta_0} J(t, u(t) ; \overline{X})
\
\text{and}\ 
\|u(t)\|_{Y_1} \leq  C_1 t^{-\theta_1} J(t, u(t) ; \overline{X}).
\]
The latter is equivalent to  the inequality
\[
t^{\theta_1-\theta_0}\|u(t)\|_{Y_1} \leq  C_1 t^{-\theta_0} J(t, u(t) ; \overline{X}).
\]
This implies that for $C=\max\{C_0, C_1\}$,  we have 
$
J(t^{\theta_1 -\theta_2}, u(t), \overline{Y}) \leq  C  t^{-\theta_0}J(t, u(t); \overline{X})$.
That is, 
\[
t^{-1+\theta_0}J(t^{\theta_1 -\theta_0}, u(t) ; \overline{Y}) \leq  C  t^{-1}J(t, u(t); \overline{X}).
\]
Since $\Phi$ satisfies the $\Delta_2$-condition,  we conclude as before that
\begin{align*}
\int_{0}^\infty  \Phi\big[ t^{-1+\theta_0}J(t^{\theta_1 -\theta_0}, u(t); \overline{Y})\big]\ dt &\leq \int_0^\infty \Phi\big[  C  t^{-1}J(t, u(t); \overline{X}) \big]\ dt\\
&\lesssim \int_0^\infty \Phi\big[    t^{-1}J(t, u(t); \overline{X}) \big]\ dt.
\end{align*}
As noted in the previous proposition, the constant involved depends only on $\Phi$, $\theta_0$, and $\theta_1$.
\end{proof}

\medskip

In preparation for the next proposition,  let us review some basic facts  about the   following  classical operators.
For $f \in L_0(0,\infty)$, we  define the  \emph{Calder\'on's operators} by setting for  $1\leq p<q<\infty$,
\[
S_{p,q}f(t) =t^{-\frac{1}{p}} \int_{0}^t  s^{\frac{1}{p}}f(s)\  \frac{ds}{s} +
t^{-\frac{1}{q}} \int_{t}^\infty s^{\frac{1}{q}}f(s)\  \frac{ds}{s}, \quad t>0
\]
and  for $1\leq p<\infty$,
\[
S_{p,\infty} f(t)=t^{-\frac{1}{p}} \int_{0}^t  s^{\frac{1}{p}}f(s)\  \frac{ds}{s} \quad  t>0.
\]

Connections between Calder\'on operators and interpolation theory are well-established in the literature.  It was noted in \cite[Proposition~5.5]{BENSHA} that 
for $1\leq p< q\leq \infty$, the linear operator $S_{p,q}$  is simultaneously of weak-types $(p,p)$ and $(q,q)$. Thus,  by standard use of Marcinkiewicz  interpolation, we have the following well-known properties:
\begin{lemma}\label{calderon} 
\begin{itemize}
\item[(i)]  For every $1\leq p< r <q$, $S_{p, q}$ is a  bounded  linear operator on $L_r(0,\infty)$;
\item[(ii)] for $1\leq p<r \leq \infty$, $S_{p,\infty}$ is  a bounded   linear operator on $L_r(0,\infty)$.
\end{itemize}
\end{lemma}
As  immediate consequences,  we also have the following $\Phi$-moment versions: 
\begin{lemma}\label{Phi} If $1 \leq  p < p_\Phi  <q_\Phi < q< \infty$, then for every $f\in L_\Phi(0,\infty)$,
\[
\int_0^\infty  \Phi\big[ |S_{p,\infty} f(t)| \big]\ dt \lesssim_{\Phi, p} \int_0^\infty \Phi[ |f(t)|] \ dt
\]
and
\[
\int_0^\infty  \Phi\big[|S_{p,q} f(t)| \big]\ dt \lesssim_{\Phi,p,q} \int_0^\infty \Phi[|f(t)|] \ dt.
\]
\end{lemma}
\begin{proof} 
From Lemma~\ref{calderon},   both $S_{p,\infty}$ and $S_{p,q}$  are  bounded  simultaneously on   $L_{r_1}(0,\infty)$ and $L_{r_2}(0,\infty)$ whenever  $p<r_1 <p_\Phi \leq q_\Phi < r_2<q$.  The two inequalities  as stated follow immediately from applying Theorem~\ref{Bekjan-Chen} to the abelian von Neumann algebra $L_\infty(0,\infty)$.
\end{proof}
The next result is a   weighted version of the previous lemma. We only consider the special case that we will use.
\begin{lemma}\label{Phi2} Let $1<p<p_\Phi \leq q_\Phi <q< \infty$.
If $g$ is a  nonnegative  decreasing function defined in $(0,\infty)$ with $t \mapsto t^{-1/p}g(t^{1/p-1/q})$ belongs to $L_\Phi(0,\infty)$, then 
\[
\int_{0}^\infty \Phi\big[ t^{-1/q} S_{1,\infty}g(t^{1/p-1/q})\big]\ dt \simeq_{\Phi,p,q}  
\int_{0}^\infty \Phi\big[ t^{-1/q} g(t^{1/p-1/q})\big]\ dt.
\]
\end{lemma}
\begin{proof} Since $g\leq S_{1,\infty}g$, one inequality is immediate.
For the non trivial inequality, let $\theta=1/p-1/q$ and define the function
\begin{equation*}
\psi(t)=t^{-1/q} S_{1,\infty}g(t^\theta) =t^{-1/p}  \int_0^{t^\theta} g(s) \ ds, \quad t>0.
\end{equation*}
Using the substitution  $s=w^\theta$, we have
\begin{align*}
\psi(t)&= \theta t^{-1/p}\int_0^t  g(w^\theta) w^{\theta -1} \ dw\\
&=\theta t^{-1/p} \int_0^t w^{1/p}  w^{-1/q}g(w^\theta) \  \frac{dw}{w}\\
&=\theta S_{p,\infty}(h_\theta) (t)
\end{align*}
where $h_\theta$ is the function $t \mapsto t^{-1/q} g(t^\theta)$. We may deduce that
\begin{align*}
\int_0^\infty \Phi\big[ t^{-1/q} S_{1,\infty}g(t^\theta)\big]\ dt &\leq  \int_0^\infty \Phi\big[ \theta S_{p,\infty}(h_\theta)(t)]\ dt \\
&\leq \int_0^\infty \Phi\big[  S_{p,\infty}(h_\theta)(t)]\ dt \\
&\lesssim \int_0^\infty \Phi[ h_\theta(t)] \ dt,
\end{align*}
where the last inequality  comes from  the first inequality in Lemma~\ref{Phi}. This is the desired inequality.
\end{proof}

\bigskip

We  now state the following weighted comparison between  $K$-functionals and $J$-functionals.

\begin{proposition}\label{KJ-comparison}  Assume that $1<p<p_\Phi \leq q_\Phi < q<\infty$ and $\overline{Y}$ is an interpolation couple. Then for every  $y  \in \Sigma(\overline{Y})$,
 \[
 \int_0^\infty \Phi\Big[t^{-1/p}K\big(t^{1/p-1/q}, y; \overline{Y}\big)\Big] \ dt
 \lesssim_{\Phi,p,q} \inf\Big\{\int_0^\infty \Phi\Big[ t^{-1/p} J\big(t^{1/p-1/q}, u(t^{1/p-1/q}); \overline{Y}\big) \Big]\ dt \Big\}
 \]
where the infimum is taken over all representations  $u(\cdot)$ of $y$.
\end{proposition}
\begin{proof} We will deduce the inequality in two steps. First, we recall the  notion of $j$-functional related to the  interpolation couple $\overline{Y}$. Suppose that $y\in  \Sigma(\overline{Y})$ admits a representation $u(\cdot)$. We define
\[
j(s,u)=j(s,u; \overline{Y})= \int_s^\infty t^{-1} J(t, u(t))\  dt/t, \quad s>0.
\]
We will verify first that  the inequality  stated in the proposition holds for $j$-functional in place of $J$-functional. That is, we claim that
\begin{equation}\label{step1}
\int_0^\infty \Phi\big[t^{-1/p}K(t^{1/p-1/q}, y)\big]\ dt \lesssim
\inf\Big\{\int_0^\infty \Phi\big[ t^{-1/q} j(t^{1/p-1/q}, u) \big]\ dt \Big\}
\end{equation}
where the infimum is taken over all representations $u(\cdot)$ of $y$ in the  couple $\overline{Y}$.

To prove this assertion,  we fix  a representation $u(\cdot)$ of $y$. As above, we let $\theta=1/p-1/q$. The crucial point of the argument  is given by the following inequality:
\[
K(t^\theta, y)   \leq \int_0^{t^\theta} j(s,u)\ ds, \quad t>0.
\]
A verification of this fact can be found for instance in \cite[p.~427]{Bennett1}.
 Since  $t^{-1/p}=t^{-1/q} t^{-\theta}$, the preceding inequality can be rewritten in the following form:
\[
t^{-1/p} K(t^\theta,y)\leq t^{-1/q} S_{1,\infty}(j(\cdot, u))(t^\theta), \quad  t>0.
\]
Since $j(\cdot, u)$ is a decreasing function,  after applying the function $\Phi$ on both sides of  the preceding inequality and  taking integrals, \eqref{step1} follows immediately from  Lemma~\ref{Phi2}.

Next, we will verify that for any representation $u(\cdot)$ of $y$, we have
\begin{equation}\label{step2}
\int_0^\infty \Phi\big[ t^{-1/q} j(t^{1/p-1/q}, u ) \big]\ dt \lesssim
\int_0^\infty \Phi\big[ t^{-1/p} J(t^{1/p-1/q}, u(t^{1/p-1/q})) \big]\ dt.
\end{equation}
Indeed, from the definition of $j(\cdot, u)$, we have  $j(t^\theta,u)= \int_{t^\theta}^\infty s^{-1}J(s, u(s)) \ ds/s$. Therefore, for every $t>0$,
\begin{equation*}
t^{-1/q} j(t^{\theta}, u)= t^{-1/q} \int_{t^\theta}^\infty  s^{-1}J(s, u(s)) \ ds/s.
\end{equation*}
Using  the substitution $s=w^\theta$,  the preceding  equality gives for every $t>0$,
\begin{align*}
t^{-1/q} j(t^{\theta}, u) &=\theta  t^{-1/q} \int_t^\infty  J(w^\theta, u(w^\theta)) w^{-2\theta} w^{\theta-1}\ dw \\
&\leq t^{-1/q} \int_t^\infty  J(w^\theta, u(w^\theta)) w^{-\theta} \ dw/w\\
&\leq t^{-1/q} \int_t^\infty  w^{1/q}  w^{-1/p}J(w^\theta, u(w^\theta))  \ dw/w\\
&\leq  S_{p,q} (\psi_\theta)(t) 
\end{align*}
where $\psi_\theta(t)=t^{-1/p} J(t^\theta, u(t^\theta))$. We deduce that
\[
\int_0^\infty  \Phi\big[ t^{-1/q} j(t^\theta, u) \big]\ dt \leq \int_0^\infty  \Phi\big[ S_{p,q}(\psi_\theta)(t)\big]\ dt \lesssim  \int_0^\infty \Phi\big[ \psi_\theta(t)\big] \ dt
\]
where the second inequality comes from the second inequality in Lemma~\ref{Phi}. 
This is  the desired  inequality. Combining \eqref{step1} and \eqref{step2} clearly gives the proposition.
\end{proof}

\begin{remark}  By choosing a representation $u(\cdot)$  satisfying $J(t, u(t)) \leq CK(t, y)$ (for some absolute constant $C$), 
the converse  of the inequality   stated in  Proposition~\ref{KJ-comparison} clearly holds but this fact 
will not be needed.
\end{remark}

We conclude this section with a discretization of the second integral appearing in  Proposition~\ref{KJ-comparison}.

\begin{lemma}\label{discretization} 
Let $1<p<q<\infty$ and  set $\theta=1/p-1/q$. Fix $y \in \Sigma(\overline{Y})$.
 \begin{itemize} 
\item[(i)] Assume that   $y=\int_0^\infty u(t) \ dt/t$  is a representation of $y$.  If  for every $\nu \in \mathbb{Z}$, we set $u_\nu=\int_{2^\nu}^{2^{\nu+1}} u(t)\ dt/t$,  then $y=\sum_{\nu \in \mathbb{Z}} u_\nu$ is a discrete representation of $y$ and
\[
\sum_{\nu \in \mathbb{Z}} 2^{\nu/\theta} \Phi\Big[ 2^{-{\nu}/{(\theta p)}} J\big(2^\nu, u_\nu; \overline{Y}\big)\Big] \lesssim_{\Phi,p,q}
\int_0^\infty  \Phi\Big[ t^{-1/p} J\big(t^{1/p-1/q}, u(t^{1/p-1/q}); \overline{Y}\big)\Big] \ dt.
\]

\item[(ii)] Conversely,  assume that $y$ admits a discrete representation $y=\sum_{\nu \in \mathbb{Z}} u_\nu$.  If we set for $t\in [2^\nu, 2^{\nu +1})$, $u(t)=u_\nu/(\log 2)$ then $y=\int_0^\infty u(t)\ dt/t$  is a representation of $y$ and
\[
\int_0^\infty  \Phi\Big[ t^{-1/p} J\big(t^{1/p-1/q}, u(t^{1/p-1/q}); \overline{Y}\big)\Big] \ dt
\lesssim_{\Phi,p,q}
\sum_{\nu \in \mathbb{Z}} 2^{\nu/\theta} \Phi\Big[ 2^{-{\nu}/{(\theta p)}} J\big(2^\nu, u_\nu; \overline{Y}\big)\Big].
\]
\end{itemize}
\end{lemma}
\begin{proof}[Sketch of the proof]
 Fix a  representation $u(\cdot)$ of $y$. A simple use of substitution gives, 
\[
\int_0^\infty  \Phi\Big[ t^{-1/p} J\big(t^{1/p-1/q}, u(t^{1/p-1/q})\big)\Big] \ dt 
=\theta^{-1}\int_0^\infty  \Phi\Big[ t^{-1/(\theta p)} J\big(t, u(t)\big)\Big]  t^{\theta^{-1} }\ dt/t. \]
Using the integral in the right hand side of the above equality, the verification of the two inequalities in the lemma is a simple adaptation of standard arguments from interpolation theory which we leave for the reader. 
\end{proof}


\section{$\Phi$-moment versions of Burkholder inequalities}

In this section, we  present our  primary objective. That is,  to formulate  $\Phi$-moment extensions of the noncommutative Burkholder inequalities.
The following theorem is the main result of this paper. It extends the noncommutative Burkholder inequalities (for the case $1<p<2$) from \cite[Theorem 6.1]{JX} to moments inequalities involving Orlicz functions.

\begin{theorem}\label{main}  Let $\Phi$ be an Orlicz function satisfying $1<p_\Phi \leq q_\Phi <2$.
There exist  positive constants $\delta_\Phi$  and $\eta_\Phi$ depending only on $\Phi$ such that  for every  martingale $x \in L_\Phi(\M)$, the following inequalities hold:
\begin{equation*}
\delta_\Phi^{-1} S_\Phi(x) \leq \T\big[\Phi(|x|)\big] \leq \eta_\Phi S_\Phi(x) \tag{$B_\Phi$}
\end{equation*}
where $S_\Phi(x)=\inf\Big\{ \T\big[ \Phi( s_c( x^c))\big] +  \T\big[ \Phi( s_r( x^r))\big]  + \sum_{n\geq 1} \T\big[\Phi(|dx_n^d|)\big] \Big\}$
 with  the infimum  being taken over all  $x^c \in \h_\Phi^c(\M)$, $x^r \in \h_\Phi^r(\M)$, and $x^d \in \h_\Phi^d(\M)$ such that $x= x^c + x^r + x^d$.
\end{theorem}

\medskip

Throughout the proof, we fix $p$ and $q$ such that $1<p<p_\Phi\leq q_\Phi<q<2$.
First, we prove the second inequality of $(B_\Phi)$. This  will be deduced  from interpolating the noncommutative Burkholder inequalities. Indeed,  since $1<p,q<2$, the noncommutative Burkholder inequalities implies that   for $s\in\{d,c,r\}$, 
$\h_p^s(\M) \subset L_p(\M)$ and $\h_q^s(\M) \subset L_q(\M)$. By Proposition~\ref{h-interpolation}, it follows that
for every $y \in \h_\Phi^d(\M)$, we have
\begin{equation}\label{diagonal}
\T\big[\Phi(|y|)\big] \leq C_\Phi \sum_{n\geq 1} \T\big[\Phi(|dy_n|)\big].
\end{equation}
Similarly, for  $z \in \h_\Phi^c(\M)$, we have
\begin{equation}\label{column}
\T\big[\Phi(|z|)\big] \leq C_\Phi' \T\big[\Phi(s_c(z))\big].
\end{equation}
Considering adjoint operators, we may also state   that  for  $w \in \h_\Phi^r(\M)$, we have
\begin{equation}\label{row}
\T\big[\Phi(|w|)\big] \leq C_\Phi' \T\big[\Phi(s_r(w))\big].
\end{equation}
Now,  let $x =x^c +x^r + x^d$ with $x^c \in \h_\Phi^c(\M)$, $x^r \in \h_\Phi^r(\M)$, and $x^d \in \h_\Phi^d(\M)$. We deduce from  \eqref{column}, \eqref{row}, and \eqref{diagonal} that
\begin{align*}
\T\big[\Phi(|x|)\big]  &\leq C_\Phi^"\Big\{\T\big[\Phi(|x^d|)\big] +\T\big[\Phi(|x^c|)\big]+\T\big[\Phi(|x^r|)\big]\Big\} \\
&\leq C_\Phi^" \max\{C_\Phi, C_\Phi'\}\Big\{\T\big[\Phi(s_c(x^c))\big]+\T\big[\Phi(s_r(x^r))\big] + \sum_{n\geq 1} \T\big[\Phi(|dx_n^d|)\big]\Big\}.
\end{align*} 
Taking the infimum over all such decompositions completes the proof of the second inequality of $(B_\Phi)$.

\medskip

Now,  we proceed with  the proof of  the first inequality of $(B_\Phi)$.
The proof will be done in several steps and rests  upon the fact  noted earlier that the Orlicz space  $L_\Phi(\M)$  is an interpolation space for the compatible couple $(L_p(\M),L_q(\M))$. A fortiori, it is an interpolation space for the compatible couple $(L_1(\M), \M)$.
 Our approach was motivated by the following  formula on  $K$-functionals: for $x\in L_1(\M) +\M$, 
\[
K(t,x; L_1(\M), \M)=\int_0^t \mu_s(x)\ ds, \quad t>0.
\]
This fact can be found for instance in \cite[Corollary~2.3]{PX3}. We make the following crucial observation:
\begin{equation}\label{Obs}
\T\big[ \Phi\big(|x|\big) \big] =\int_0^\infty \Phi(\mu_t(x))\ dt \simeq \int_0^\infty  \Phi\Big[ t^{-1} K(t, x)\Big]\ dt,
\end{equation}
where the equivalence is taken from  the property of Calder\'on's  operator stated in Lemma~\ref{Phi}. Thus, proving the first inequality in $(B_\Phi)$ amounts to  finding  suitable  estimate for the integral of  the function $t\mapsto \Phi\big[t^{-1}K(t,x)\big]$ from below. However, as it
 will be clear from the steps  taken below,  the  $J$-functionals  computed  with respect to   the compatible 
couple $(L_p(\M), L_q(\M))$  turn out to be  the right framework for  this stated goal.  Below, $C_{\Phi, p,q}$ denotes a  positive constant whose value may change from one line to the next.

\medskip

 $\bullet$ We assume first that $x \in L_1(\M) \cap \M$.
 
\noindent{Step 1.}      Choose a representation  $u(\cdot)$ of $x$ in the compatible couple $(L_1(\M), \M)$ such that:

\begin{equation}\label{J-K}
J(t, u(t)) \leq C K(t,x), \quad t>0
\end{equation}
where $C$ is an absolute constant.
Thus, since $\Phi$ has the $\Delta_2$-condition, we have from \eqref{Obs} and  \eqref{J-K}  that
\begin{equation}\label{J-L1}
\int_0^\infty  \Phi\Big[ t^{-1} J\big(t, u(t)\big)\Big]\ dt  \leq C_\Phi \T\big[ \Phi\big(|x|\big) \big].
\end{equation}

\medskip

\noindent{Step~2.} {\it Changing into  the compatible  couple $(L_p(\M), L_q(\M))$}.  This is done through Proposition~\ref{J-iteration}. Denote by $\overline{X}$ the compatible couple $(L_1(\M), \M)$. If we set $\theta_0=1-p^{-1}$ and $\theta_1=1-q^{-1}$ then $L_p(\M)$ and $L_q(\M)$ belong to  the classes $\cal{C}_J(\theta_0,\overline{X})$ and $\cal{C}_J(\theta_1,\overline{X})$, respectively. 
We claim that $u(\cdot)$ is also a representation $x$  for the compatible couple $(L_p(\M), L_q(\M))$. To verify this claim, fix $p<r<q$. Since $x \in L_1(\M) \cap \M$, it  belongs to $L_r(\M)$. Let $\theta=1-r^{-1}$. We recall  that $L_r(\M)=(L_1(\M),\M)_{\theta,r, K}$ (with equivalent norms) where $(\cdot, \cdot)_{\theta, r, K}$ denotes the real interpolation using the $K$-method (see \cite{BL}). We have  by  the definition of  $(L_1(\M),\M)_{\theta,r, K}$ that the function 
$
t^{-\theta} K(t, x; \overline{X}) \in L_r(\mathbb{R}_+, dt/t)$.
From \eqref{J-K}, we also have $t^{-\theta} J(t, u(t);\overline{X}) \in L_r(\mathbb{R}_+, dt/t)$. It  is implicit in the proof of \cite[Proposition~3.3.19, p~177-178]{Butzer-Berens}
that the latter assertion implies that  the integral
$\int_0^\infty u(t)\ dt/t$ is convergent in $L_p(\M) +L_q(\M)$. This verifies the claim.

With the above observations, it is clear that Proposition~\ref{J-iteration} applies to our situation. We should point out here that the only reason for considering $x \in L_1(\M) \cap \M$ is to insure that $u(\cdot)$ is a representation of $x$ for both compatible couples.

 Putting   \eqref{J-L1} together with Proposition~\ref{J-iteration} yield:
\begin{equation*}
\int_0^\infty  \Phi\Big[ t^{-1/p} J\big(t^{1/p-1/q}, u(t); L_p(\M), L_q(\M)\big)\Big]\ dt  \leq C_{\Phi ,p,q}\T\big[ \Phi\big(|x|\big) \big].
\end{equation*}
For technical reasons that should be clear later, we need to modify the representation as follows: set $1/\alpha=1/p -1/q$ and define:
\[v(t)=\alpha u(t^\alpha)\quad \text{ for $t>0$}.
\]
A simple use of substitution  shows that $v(\cdot)$ is a representation of $x$ in the compatible  couple
$(L_1(\M), \M)$ (also   for the compatible couple $(L_p(\M), L_q(\M))$). Using the representation $v(\cdot)$, 
the preceding inequality becomes:
\begin{equation}\label{J-p-v}
\int_0^\infty  \Phi\Big[ t^{-1/p} J\big(t^{1/p-1/q}, v(t^{1/p-1/q}); L_p(\M), L_q(\M)\big)\Big]\ dt  \leq C_{\Phi ,p,q}\T\big[ \Phi\big(|x|\big) \big].
\end{equation}
 Next, we discretize the integral in \eqref{J-p-v}. If we set $v_\nu=\int_{2^\nu}^{2^{\nu +1}} v(t)\ dt/t$ for every $\nu \in \mathbb{Z}$,  then  $v_\nu \in L_1(\M) \cap \M$ and 
 \begin{equation}\label{discrete-v}
 x=\sum_{\nu \in \mathbb{Z}} v_\nu \ \text{(convergence in $L_p(\M) + L_q(\M)$)}.
 \end{equation}
By Lemma~\ref{discretization}(i), we  deduce from \eqref{J-p-v} that if $\theta=1/p-1/q$ then 
\begin{equation}\label{J-p-v-discrete}
\sum_{\nu \in \mathbb{Z}} 2^{\nu/\theta}  \Phi\Big[2^{-\nu/(\theta p)} J\big(2^\nu, v_\nu; L_p(\M), L_q(\M)\big) \Big]  \leq C_{\Phi ,p,q}\T\big[ \Phi\big(|x|\big) \big].
\end{equation}

\medskip

\noindent{Step~3.}  {\it Use  of the simultaneous decompositions}. In this step, we use   the simultaneous decomposition to generate the appropriate decomposition of $x$. 
This is a reminiscent of an argument used in \cite{Ran18} (see also \cite{Ran21,RW}).

For each $\nu \in \mathbb{Z}$, we note  that  since $v_\nu \in L_1(\M) \cap \M$,  Theorem~\ref{simultaneous}  applies to $v_\nu$.  That is, there exist $a_\nu$, $b_\nu$, and $c_\nu$ in $L_p(\M) \cap L_q(\M)$ satisfying:
\begin{equation}\label{decompositions}
v_\nu = a_\nu + b_\nu + c_\nu
\end{equation}
and if $s$ is equal to either $p$ or $q$,  then
\begin{equation}\label{double-decomposition}
\big\|a_\nu \big\|_{\h_s^d} + \big\|b_\nu \big\|_{\h_s^c} +\big\|c_\nu \big\|_{\h_s^r} \leq 
\kappa(p,q) \big\| v_\nu \big\|_s
\end{equation}
where $\kappa(p,q)=\max\{\kappa_p, \kappa_q\}$ with $\kappa_p$ and $\kappa_q$ are constants from Theorem~\ref{simultaneous}.
For convenience,  we let 
 \[\overline{A} :=(L_p(\M\overline{\otimes} \ell_\infty), L_q(\M\overline{\otimes} \ell_\infty))\  \text{and}\    \overline{B} :=(L_p(\M \overline{\otimes} B(\ell_2(\mathbb{N}^2))) , L_q(\M \overline{\otimes} B(\ell_2(\mathbb{N}^2)))).
 \]

For any given  $\nu \in \mathbb{Z}$, we consider the sequences  $\mathcal{D}_d(a_\nu) \in \Delta(\overline{A})$, 
$U\mathcal{D}_c(b_\nu) \in \Delta(\overline{B})$,  and $U\mathcal{D}_c(c_\nu^*) \in \Delta(\overline{B})$. 
We make the crucial observation  that  the inequalities in  \eqref{double-decomposition} can be reinterpreted using the $J$-functionals as follows:
\begin{equation}\label{J-inequalities}
\begin{split}
J\big( t, \mathcal{D}_d(a_\nu); \overline{A}) &\leq \kappa(p,q) J( t, v_\nu; L_p(\M), L_q(\M)),\  
t>0,\\
 J\big(t, U\mathcal{D}_c(b_\nu); \overline{B}\big) &\leq \kappa(p,q) J( t, v_\nu; L_p(\M), L_q(\M)), \ t>0,
\\
J\big(t, U\mathcal{D}_c(c_\nu^*); \overline{B}\big) &\leq \kappa(p,q) J( t, v_\nu; L_p(\M), L_q(\M)), \ t>0.
\end{split}
\end{equation}

We need the following properties of the  sequences $\{\mathcal{D}_d(a_\nu)\}_{\nu \in \mathbb{Z}}$, 
$\{U\mathcal{D}_c(b_\nu)\}_{\nu \in \mathbb{Z}}$, and $\{U\mathcal{D}_c(c_\nu^*)\}_{\nu \in \mathbb{Z}}$. 
\begin{sublemma}\label{sublemma} 
\begin{enumerate}[{\rm(1)}]
 \item  $ \sum_{\nu \in \mathbb{Z}}  \mathcal{D}_d(a_\nu)$  is   (unconditionally)   convergent in 
 $L_\Phi(\M \overline{\otimes} \ell_\infty)$.
\item $ \sum_{\nu \in \mathbb{Z}} U\mathcal{D}_c(b_\nu)$  is  (unconditionally) convergent in  $L_\Phi(\M \overline{\otimes} B(\ell_2(\mathbb{N}^2))) $. 
\item $\sum_{\nu \in \mathbb{Z}} U\mathcal{D}_c(c_\nu^*)$ is  (unconditionally) convergent in  $L_\Phi(\M \overline{\otimes} B(\ell_2(\mathbb{N}^2)))$.
\end{enumerate}
\end{sublemma}
The proof of  Sublemma~\ref{sublemma} is identical to that of \cite[Sublemma~3.3]{RW}. Indeed, the argument used in \cite{RW}  would show that these series are weakly unconditionally Cauchy but since $L_\Phi$ is reflexive these convergences are automatically unconditional (see \cite{D1}). We leave the details to the reader.  From Sublemma~\ref{sublemma}, we  may deduce that there exist $a \in \h_\Phi^d(\M)$, $b \in \h_\Phi^c(\M)$, and $c \in \h_\Phi^r(\M)$ such that:
\begin{equation}\label{sum}
\begin{split}
\mathcal{D}_d(a) &:= \sum_{\nu \in \mathbb{Z}} \mathcal{D}_d(a_\nu)  \in L_\Phi(\M \overline{\otimes} \ell_\infty);\\
U\mathcal{D}_c(b) &:=\sum_{\nu \in \mathbb{Z}} U\mathcal{D}_c(b_\nu) \in L_\Phi(\M \overline{\otimes} B(\ell_2(\mathbb{N}^2))); \\
U\mathcal{D}_c (c^*) &:=\sum_{\nu \in \mathbb{Z}} U\mathcal{D}_c(c_\nu^*) \in L_\Phi(\M \overline{\otimes} B(\ell_2(\mathbb{N}^2))).
\end{split}
\end{equation}
The fact that the sum of the first series belongs to $\mathcal{D}_d(\h_\Phi^d(\M))$
is clear since the terms of the series belong to  the closed subspace $\mathcal{D}_d(\h_\Phi^d(\M))$ and thus the existence of $a \in \h_\Phi^d(\M)$. Similar observations can be made for the other two series.
Now,  combining \eqref{J-p-v-discrete} with \eqref{J-inequalities} lead to the following inequalities:
\begin{equation}\label{J2-discrete}
\begin{split}
\sum_{\nu \in \mathbb{Z}} 2^{\nu/\theta}  \Phi\Big[2^{-\nu/(\theta p)} J(2^\nu, \mathcal{D}_d(a_\nu); \overline{A}) \Big]  &\leq C_{\Phi ,p,q}\T\big[ \Phi\big(|x|\big) \big];\\
\sum_{\nu \in \mathbb{Z}} 2^{\nu/\theta}  \Phi\Big[2^{-\nu/(\theta p)} J(2^\nu, U\mathcal{D}_c(b_\nu); \overline{B}) \Big]  &\leq C_{\Phi ,p,q}\T\big[ \Phi\big(|x|\big) \big];\\
\sum_{\nu \in \mathbb{Z}} 2^{\nu/\theta}  \Phi\Big[2^{-\nu/(\theta p)} J(2^\nu, U\mathcal{D}_c(c_\nu^*); \overline{B}) \Big]  &\leq C_{\Phi ,p,q}\T\big[ \Phi\big(|x|\big) \big].
\end{split}
\end{equation}
Next, we go back to the continuous case. By setting for $t\in [2^\nu, 2^{\nu +1})$,
\[
\mathcal{D}_d(a(t))=\frac{\mathcal{D}_d(a_\nu)}{\log2} \in \Delta(\overline{A}),\  U\mathcal{D}_c(b(t))=\frac{U\mathcal{D}_c(b_\nu)}{\log2} \in \Delta(\overline{B}), \ \text{and}\ U \mathcal{D}_c(c(t)^*)=\frac{U\mathcal{D}_c(c_\nu^*)}{\log2} \in \Delta(\overline{B}),
\]
we obtain that $\mathcal{D}_d(a(\cdot))$ is a representation of $\mathcal{D}_d(a)$ in the couple $\overline{A}$, $U\mathcal{D}_c(b(\cdot))$ is a representation of $U\mathcal{D}_c(b)$ in the couple $\overline{B}$, and $U \mathcal{D}_c(c(\cdot)^*)$ is a representation of $U\mathcal{D}_c(c^*)$ in the couple $\overline{B}$.  Moreover, Lemma~\ref{discretization}(ii) and \eqref{J2-discrete} give  integral estimates involving the $J$-functionals:
\begin{equation}\label{J2-continuous}
\begin{split}
\int_0^\infty \Phi\Big[ t^{-1/p} J\big(t^{1/p-1/q}; \mathcal{D}_d(a(t^{1/p-1/q})); \overline{A} \big) \Big]\ dt &\leq C_{\Phi ,p,q}\T\big[ \Phi\big(|x|\big) \big],\\
\int_0^\infty \Phi\Big[ t^{-1/p} J\big(t^{1/p-1/q}, U\mathcal{D}_c(b(t^{1/p-1/q})); \overline{B} \big) \Big]\ dt &\leq C_{\Phi ,p,q}\T\big[ \Phi\big(|x|\big) \big],\\
\int_0^\infty \Phi\Big[ t^{-1/p} J\big(t^{1/p-1/q}, U\mathcal{D}_c(c(t^{1/p-1/q})^*); \overline{B} \big) \Big]\ dt &\leq C_{\Phi ,p,q}\T\big[ \Phi\big(|x|\big) \big].
\end{split}
\end{equation}

\medskip

\noindent{Step~4.} {\it Switching  back to $K$-functionals}. In this step, we rewrite \eqref{J2-continuous} using $K$-functionals. Indeed, from Proposition~\ref{KJ-comparison}, we  may state that:
\begin{equation}\label{K2-continuous}
\begin{split}
\int_0^\infty \Phi\Big[ t^{-1/p} K\big(t^{1/p-1/q}, \mathcal{D}_d(a); \overline{A} \big) \Big]\ dt &\leq C_{\Phi ,p,q}\T\big[ \Phi\big(|x|\big) \big],\\
\int_0^\infty \Phi\Big[ t^{-1/p} K\big(t^{1/p-1/q}, U\mathcal{D}_c(b); \overline{B} \big) \Big]\ dt &\leq C_{\Phi ,p,q}\T\big[ \Phi\big(|x|\big) \big],\\
\int_0^\infty \Phi\Big[ t^{-1/p} K\big(t^{1/p-1/q},  U\mathcal{D}_c(c^*) ; \overline{B} \big) \Big]\ dt &\leq C_{\Phi ,p,q}\T\big[ \Phi\big(|x|\big) \big].
\end{split}
\end{equation}

The final part  of the argument is to convert the inequalities in \eqref{K2-continuous} to the  $(L_1,L_\infty)$ type interpolation couples.  This necessary since our initial connections with $\Phi$-moments are with the $K$-functionals relative to the couple $(L_1,L_\infty)$.
 We use Proposition~\ref{K-iteration} for this task. We recall  that if $\N$ is an arbitrary  semifinite von Neumann algebra equipped with a normal semifinite trace  and $\theta_0=1-p^{-1}$ and $\theta_1=1-q^{-1}$ then 
$L_p(\N)$ and $L_q(\N)$ belong to the class $\cal{C}_K(\theta_0, (L_1(\N), \N))$ and
$\cal{C}_K(\theta_0, (L_1(\N), \N))$, respectively. Thus, if  we set $\N_1 :=\M\overline{\otimes} \ell_\infty$ and $\N_2 := \M\overline{\otimes} B(\ell_2(\mathbb{N}^2))$,  then  we may deduce from  \eqref{K2-continuous} and Proposition~\ref{K-iteration} that:
\begin{equation}\label{last}
\begin{split}
\int_0^\infty \Phi\Big[ t^{-1} K\big(t, \mathcal{D}_d(a) ; L_1(\N_1), \N_1 \big) \Big]\ dt &\leq C_{\Phi ,p,q}\T\big[ \Phi\big(|x|\big) \big],\\
\int_0^\infty \Phi\Big[ t^{-1} K\big(t, U\mathcal{D}_c(b); L_1(\N_2),\N_2 \big) \Big]\ dt &\leq C_{\Phi ,p,q}\T\big[ \Phi\big(|x|\big) \big],\\
\int_0^\infty \Phi\Big[ t^{-1} K\big(t, U\mathcal{D}_c(c^*); L_1(\N_2), \N_2 \big) \Big]\ dt &\leq C_{\Phi ,p,q}\T\big[ \Phi\big(|x|\big) \big].
\end{split}
\end{equation}

\medskip

\noindent{Step~5.} {\it Converting \eqref{last} into $\Phi$-moment inequalities.}
For this,   we consider first the diagonal part. We observe  that if $\T \otimes \gamma$ denotes the  natural trace of $\N_1$  and $a=\sum_n da_n$, then 
\begin{align*}
\sum_{n\geq 1} \T\big( \Phi( |da_n|) \big) &=\T\otimes \gamma\big[ \Phi(|\mathcal{D}_d(a)|)\big]\\
&=
\int_0^\infty \Phi\big[ \mu_t(\mathcal{D}_d(a))\big] \ dt\\
&\leq \int_0^\infty \Phi\Big[ t^{-1} K\big(t, \mathcal{D}_d(a); L_1(\N_1), \N_1 \big) \Big]\ dt \\
&\leq C_{\Phi ,p,q}\T\big[ \Phi\big(|x|\big) \big],
\end{align*}
where the singular-value in the 2nd line is taken with respect to $(\N_1, \T \otimes\gamma)$ and the last inequality  comes  from \eqref{last}. This establishes the diagonal part.

For the column version,  we have the estimates:
\begin{align*}
\T\big[ \Phi(s_c(b))\big] &=\T \otimes \Tr\big[\Phi\big(\big| U\mathcal{D}_c(b)]\big|\big)\big]\\
&=\int_0^\infty \Phi\big[\mu_t\big(U\mathcal{D}_c(b)\big)\big]\ dt\\
&\leq \int_0^\infty \Phi\Big[ t^{-1} K\big(t, U \mathcal{D}_c(b) ;  L_1(\N_2),\N_2 \big) \Big]\ dt \\
&\leq C_{\Phi ,p,q}\T\big[ \Phi\big(|x|\big) \big],
\end{align*}
where  the first equality comes from  \eqref{convension-moment}, the singular values are taken relative to  $( \N_2, \T \otimes \Tr)$, and the last inequality is from \eqref{last}.
Similarly, we may also deduce that
\[
\T\big[ \Phi(s_r(c))\big] =\T\big[ \Phi(s_c(c^*))\big]\leq C_{\Phi ,p,q}\T\big[ \Phi\big(|x|\big) \big].
\]
By combining the  last three estimates, we have
\[
\T\big[ \Phi(s_c(b))\big] + \T\big[ \Phi(s_r(c))\big] +\sum_{n\geq 1} \T\big[ \Phi( |da_n|) \big] \leq C_{\Phi ,p,q}\T\big[\Phi\big(|x|\big) \big].
\]
To conclude the proof,   we note  from \eqref{discrete-v}, \eqref{decompositions}, and \eqref{sum},  that  the identity $x=a+b+c$ is clear from the construction.  This completes the proof for the case $x \in L_1(\M) \cap \M$. 

\smallskip

$\bullet$ Assume now that $x \in L_\Phi(\M)$. Since $L_\Phi(\M)$ is a reflexive space, $L_1(\M) \cap \M$  is a dense subset of $L_\Phi(\M)$. Fix  a sequence $(x^{(m)})_{m\geq 1}$ in $L_1(\M)\cap \M$ such  that $\lim_{m\to\infty} \|x^{(m)} -x\|_{L_\Phi(\M)}=0$. By Lemma~\ref{convergence}, we also have $ \lim_{m\to \infty} \T\big[\Phi\big(|x^{(m)}|\big) \big]=  \T\big[\Phi\big(|x|\big) \big]$.
From the previous case, for every $m\geq 1$, there exists a decomposition $x^{(m)}=a^{(m)} +b^{(m)} +c^{(m)} $ with $a^{(m)} \in \h_\Phi^d(\M)$, $b^{(m)} \in \h_\Phi^c(\M)$, and $c^{(m)} \in \h_\Phi^r(\M)$ that satisfy 
\[
\T\big[ \Phi(s_c(b^{(m)}))\big] + \T\big[ \Phi(s_r(c^{(m)}))\big] +\sum_{n\geq 1} \T\big[ \Phi( |da_n^{(m)}|) \big] \leq C_{\Phi ,p,q}\T\big[\Phi\big(|x^{(m)}|\big) \big].
\]
From reflexivity, we may assume (by taking subsequence if necessary) that the sequence  of triplets $\{(a^{(m)}, b^{(m)}, c^{(m)})\}_{m \geq 1}$ converges to $(a,b,c)$ for the weak topology in $ \h_\Phi^d(\M) \oplus_\infty \h_\Phi^c(\M)
\oplus_\infty\h_\Phi^r(\M)$. Clearly, $x=a +b +c$. By Lemma~\ref{convergence}, we have
$\sum_{n\geq 1}\T\big[ \Phi( |da_n|) \big] \leq \liminf_{m\to \infty} \sum_{n\geq 1}\T\big[ \Phi( |da_n^{(m)}|) \big]$,
$\T\big[ \Phi(s_c(b))\big] \leq \liminf_{m\to\infty} \T\big[ \Phi(s_c(b^{(m)}))\big]$, and
$\T\big[ \Phi(s_r(c))\big]\leq \liminf_{m\to \infty} \T\big[ \Phi(s_r(c^{(m)}))\big]$. These yield the following estimates:
\begin{align*}
S(a,b,c;\Phi) &:=\T\big[ \Phi(s_c(b))\big] + \T\big[ \Phi(s_r(c))\big] +\sum_{n\geq 1} \T\big[ \Phi( |da_n|) \big] \\
&\leq\limsup_{m\to \infty} \Big\{ \T\big[ \Phi(s_c(b^{(m)}))\big] + \T\big[ \Phi(s_r(c^{(m)}))\big] +\sum_{n\geq 1} \T\big[ \Phi( |da_n^{(m)}|) \big]\Big\}\\
&\leq C_{\Phi, p,q} \lim_{m\to \infty} \T\big[\Phi\big(|x^{(m)}|\big) \big]\\
&=C_{\Phi,p,q}\T\big[\Phi\big(|x|\big) \big].
\end{align*}
The  proof is complete.
 \qed
 
 \begin{remark} For the case where $\M$ is a finite von Neumann algebra, it is not necessary  in  our argument above to separate the particular case where $x\in L_1(\M) \cap \M$. Indeed, when $\M$ is finite, $L_\Phi(\M) \subset  L_r(\M)$ whenever $p<r<p_\Phi$, thus the argument used in Step~2 applies directly to any  element of $L_\Phi(\M)$.
 \end{remark}


\medskip

Our next result   deals with the case  where the indices of the Orlicz function $\Phi$ are larger than $2$. It may be viewed as a  common generalization of  a $\Phi$-moment result from  classical martingale theory     \cite[Theorem ~1]{Mogy}   and    the noncommutative Burkholder's  inequalities from \cite[Theorem~5.1]{JX}.

\begin{theorem}\label{pl2}
Let $\Phi$ be an Orlicz function satisfying $2<p_\Phi \leq q_\Phi <\infty$. 
There exist  positive constants $\delta_\Phi$  and $\eta_\Phi$ depending only on $\Phi$ such that  for every  martingale $x \in L_\Phi(\M)$, the following inequalities hold:
\begin{equation*}
\delta_\Phi^{-1} M_\Phi(x) \leq \T\big[\Phi(|x|)\big] \leq \eta_\Phi M_\Phi(x) \tag{$B_\Phi$}
\end{equation*}
where $M_\Phi(x)=\max\Big\{ \sum_{n\geq 1}\T\big[ \Phi\big( |dx_n|\big) \big], \T\big[ \Phi( s_c( x))\big],\T\big[ \Phi( s_r( x))\big] \Big\}$.
\end{theorem}

\begin{proof}   We begin with  the first inequality. This  is a simple application of Proposition~\ref{h-interpolation} and  the noncommutative Burkholder inequalities. 
We  leave the details to the reader.

\medskip

The proof for the second inequality is more involved.
Our approach is a duality type argument  based on the first inequality in Theorem~\ref{main} and Proposition~\ref{duality}. Let
$\Phi^*$ denote the Orlicz function complementary  to $\Phi$. First, we note that since $2<p_\Phi \leq q_\Phi<\infty$, it follows  that
$
1<p_{\Phi^*} \leq q_{\Phi^*} <2$. In particular, Theorem~\ref{main} applies to bounded martingales in $L_{\Phi^*}(\M)$.

Next, we observe  that $\lim_{t\to 0^+}  M(t, \Phi^*)=0$. This fact can be easily seen from the definitions of  the indices. We may choose $t_\Phi$ small enough so that  \[M(t_\Phi, \Phi^*) \leq (2\delta_{\Phi^*})^{-1}
\]
where $\delta_{\Phi^*}$ is the constant from Theorem~\ref{main}  applied to $\Phi^*$. This is equivalent to 
\[
\Phi^*(t_\Phi s) \leq (2\delta_{\Phi^*})^{-1} \Phi^*(s), \quad s>0.
\]
Thus, by functional calculus, for any operator $0\leq z \in L_{\Phi^*}(\M)$, we have
\begin{equation}\label{choose-t}
\Phi^*(t_\Phi z) \leq (2\delta_{\Phi^*})^{-1} \Phi^*(z).
\end{equation}
We are now ready to provide the proof. 
 Assume first that $x \in L_1(\M) \cap \M$.
 By Proposition~\ref{duality},   we may choose $0\leq y \in L_{\Phi^*}(\M)$ such that  $y$ commutes with $|x|$ and  
\begin{equation}\label{choose-y}
 \Phi\big(|x|\big) + \Phi^*(y) =y|x|.
\end{equation}
If $x=u|x|$ is  the polar decomposition of $x$,  we set $y' :=yu^*\in L_{\Phi^*}(\M)$.
Applying Theorem~\ref{main} to $y'$, there exists a decomposition $y'=y^c +y^r + y^d$
with $y^c \in \h_{\Phi^*}^c(\M)$, $y^r \in \h_{\Phi^*}^r(\M)$,  and $y^d \in \h_{\Phi^*}^d(\M)$ satisfying:
\begin{equation}\label{decomposition-y}
\T\big[ \Phi^*( s_c( y^c))\big] +  \T\big[ \Phi^*( s_r( y^r))\big]  + \sum_{n\geq 1} \T\big[\Phi^*(|dy_n^d|)\big]  \leq 2\delta_{\Phi^*}  \T\big[\Phi^*(|y'|)\big].
\end{equation}
 Taking traces on \eqref{choose-y} together with the decomposition of $y'$,  we have
 \begin{align*}
 \T\big[\Phi(|x|)\big]  +\T\big[ \Phi^*(y)] &=  \T(y|x|)\\
&=\T(y'x)\\
&=\T(xy^d) + \T(xy^c) +\T(xy^r)\\
&:= I + II + III.
 \end{align*}
 We estimate $I$, $II$, and $III$ separately.  First, by
applying Lemma~\ref{Young} and \eqref{choose-t},  we get the following estimates:
\begin{align*}
I &=\sum_{n\geq 1} \T(dx_n d{y_n^d})\\
&\leq \sum_{n\geq 1} \T\big[ \Phi\big(t_\Phi^{-1}|dx_n|\big)\big] +\T\big[\Phi^*\big(t_\Phi |d{y_n^d}|\big)\big]\\
&\leq  \sum_{n\geq 1}\T\big[ \Phi\big(t_\Phi^{-1}|dx_n|\big)\big] + (2\delta_{\Phi^*})^{-1} \sum_{n\geq 1}\T\big[\Phi^*\big( |dy_n^d|\big)\big].
\end{align*}
To estimate $II$, we use the embedding of $\h_\Phi^c(\M)$ into $L_\Phi(\M \overline{\otimes} B(\ell_2(\mathbb{N}^2)))$.  First, we note that  since the conditional  expectations $\E_k$'s are trace preserving, we have 
\[
II =\sum_{n\geq 1} \T(\E_{n-1}( dx_nd{y_n^c}))=\T\Big(\sum_{n\geq 1} \E_{n-1}( dx_n d{y_n^c})\Big). \]
We should note here that since $x \in L_1(\M) \cap \M$, for every $n\geq 1$, $dx_n dy_n^c \in L_1(\M) + \M$ and therefore   
$\sum_{n\geq 1} \E_{n-1}( dx_nd{y_n^c})$ is a well-defined operator that belongs to $L_1(\M)$. We claim that
\[
II =\T \otimes \Tr \big[ U\mathcal{D}_c(x^*)^*  U\mathcal{D}_c(y^c)  \big].
\]
To verify this claim, we  begin with the fact  taken  from Lemma~\ref{comp-Int} that $\h_{\Phi^*}^c(\M)  \subseteq \h_p^c(\M) +\h_q^c(\M)$ where $1<p<p_{\Phi^*} \leq q_{\Phi^*} <q <2$.
Write $y^c =\alpha^c + \beta^c$ where $\alpha^c  \in \h_p^c(\M)$ and $\beta^c \in \h_q^c(\M)$. Then from \eqref{c-square},  we have
\[
\big(\sum_{n\geq 1} \E_{n-1}(dx_n d\alpha_n^c)\big) \otimes e_{1,1} \otimes e_{1,1} =
 U\mathcal{D}_c(x^*)^* U\mathcal{D}_c(\alpha^c)
\]
and
\[
\big(\sum_{n\geq 1} \E_{n-1}(dx_n d\beta_n^c)\big) \otimes e_{1,1} \otimes e_{1,1} =
 U\mathcal{D}_c(x^*)^* U\mathcal{D}_c(\beta^c).
\]
Taking the sum of the above two equalities clearly shows  the claim.

 As in the case of $I$,  by applying Lemma~\ref{Young} together with  \eqref{choose-t},  we obtain the estimates
\begin{align*}
II &=\T \otimes \Tr \big[ U\mathcal{D}_c(x^*)^*  U\mathcal{D}_c(y^c)  \big]\\
&\leq \T \otimes \Tr \Big[ \Phi\big(t_\Phi^{-1}|U\mathcal{D}_c(x^*)|\big)\Big] +
\T \otimes \Tr \Big[ \Phi^*\big(t_\Phi|U\mathcal{D}_c(y^c)|\big)\Big]\\
&=\T\big[ \Phi\big(t_\Phi^{-1} s_r(x) \big)\big] +  \T\big[ \Phi^*\big(t_\Phi s_c(y^c) \big)\big]\\
&\leq \T\big[ \Phi\big(t_\Phi^{-1} s_r(x) \big)\big]  + (2\delta_{\Phi^*})^{-1} \T\big[ \Phi^*\big(s_c(y^c) \big)\big].
\end{align*}
By repeating the same  argument with $y^r$,  we may also state  that 
\[III  \leq \T\big[ \Phi\big(t_\Phi^{-1} s_c(x) \big)\big]  + (2\delta_{\Phi^*})^{-1} \T\big[ \Phi^*\big(s_r(y^r)\big) \big].\]
Taking the summation of the previous  estimates and applying \eqref{decomposition-y}, we obtain  that
\begin{equation*}
\begin{split}
 \T\big[\Phi(|x|)\big]  + \T\big[\Phi^*(y) \big] \leq  &\sum_{n\geq 1}\T\big[ \Phi\big(t_\Phi^{-1}|dx_n|\big)\big]  + \T\big[ \Phi\big(t_\Phi^{-1} s_c(x) \big)\big] \\
&+\T\big[ \Phi\big(t_\Phi^{-1} s_r(x) \big)\big]  + \T\big[\Phi^*(|y'|) \big].
\end{split}
\end{equation*}
But since $ \T\big[\Phi^*(|y'|) \big] = \int_0^\infty \Phi^*(\mu_t(yu^*))\ dt \leq   \int_0^\infty \Phi^*(\mu_t(y))\ dt = \T\big[\Phi^*(y) \big]$, we deduce  that 
\begin{align*}
 \T\big[\Phi(|x|)\big]    &\leq \sum_{n\geq 1}\T\big[ \Phi\big(t_\Phi^{-1}|dx_n|\big)\big]  + \T\big[ \Phi\big(t_\Phi^{-1} s_c(x) \big)\big] +\T\big[ \Phi\big(t_\Phi^{-1} s_r(x) \big)\big]\\
 &\leq 3\max\Big\{\sum_{n\geq 1}\T\big[ \Phi\big(t_\Phi^{-1}|dx_n|\big)\big]  , \T\big[ \Phi\big(t_\Phi^{-1} s_c(x) \big)\big] , \T\big[ \Phi\big(t_\Phi^{-1} s_r(x) \big)\big]\Big\}.
\end{align*}
The existence of the constant $\eta_\Phi$ and the second inequality in $(B_\Phi)$  now follow from the $\Delta_2$-condition.
Thus, we have shown the second  inequality in $(B_\Phi)$ for   $x \in L_1(\M) \cap \M$.  The proof for the general case follows the same line of reasoning as in the last part of the proof of Theorem~\ref{main} so we omit the details.
\end{proof}

At the time of this writing, we do not know  of any direct proof of Theorem~\ref{pl2}.
The existing argument for $p$-th moment from \cite{JX} can be adapted to  $\Phi$-moment only for the case where  the Orlicz function $\Phi$ satisfies a H\"older type inequality   $\Phi(ts) \leq C \Phi(t^2)^{1/2}\Phi(s^2)^{1/2}$ for every $t, s>0$ and $C$ is an absolute constant.  The above condition is clearly satisfied by power functions and exponential functions $t\mapsto e^{\alpha t}$ with $\alpha>0$. It is however stronger than being submultiplicative.
We should point out that the approach used in \cite{JX} for the $p$-th moments was to establish the case $2<p<\infty$ first and  then deduce the case $1<p<2$  using duality.
We do not know if  Theorem~\ref{main} can be derived from Theorem~\ref{pl2} via Proposition~\ref{duality}.

\section{$\Phi$-moments and noncommutative Rosenthal inequalities}

In this section,   we consider  notions of noncommutative   independences  introduced in \cite{JX3} and discuss   corresponding $\Phi$-moment results for sums of independent sequences.

 Throughout, we assume that  $\cal{N}$  and $\cal{A}_n$'s are  von Neumann subalgebras  of $(\M,\T)$  with $\cal{N} \subset \cal{A}_n$ for all $n\geq 1$. We  further assume that there exist trace preserving  normal conditional expectations $\E_{\cal{N}}: \M \to \cal{N}$ and $\E_{\cal{A}_n}:\M \to \cal{A}_n$ for all $n\geq 1$. Following \cite{JX3}, we consider the following notions of independences:

\begin{definition}
(i)  We say that $(\cal{A}_n)_{n\geq 1}$  are \emph{independent}  over $\cal{N}$ (or with respect to $\E_{\N}$)  if for every $n\geq 1$, $\E_{\N}(xy)=\E_{\N}(x)\E_{\N}(y)$ holds for all $x \in \cal{A}_n$ and $y$ in the von Neumann algebra generated by $(\cal{A}_j)_{j\neq n}$.

(ii) We say that the   sequence
$(\cal{A}_n)_{n\geq 1}$ 
is   \emph{order independent over $\N$ (or with respect to $\E_{\cal{N}}$})  if for every $n\geq2$, 
\[
\E_{VN(\cal{A}_1,\dots,
\cal{A}_{n-1})}(x)=\E_{\cal{N}}(x), \  x \in \cal{A}_n
\]
where $\E_{VN(\cal{A}_1, \dots, \cal{A}_{n-1})}$ denotes the normal conditional expectation onto the von Neumann subalgebra degenerated by $\cal{A}_1,\dots, \cal{A}_{n-1}$.

(iii) A sequence $(a_n)_{n\geq 1}$  in $L_1(\M)+\M$  is called \emph{(order) independent with respect to $\E_{\cal{N}}$} if
there is an (order) independent sequence $(\cal{A}_n)_{n\geq 1}$  of
von Neumann subalgebras of $\M$ such that $a_n \in L_1(\cal{A}_n) +\cal{A}_n$
for all $n\geq 1$.
\end{definition}

It was noted in \cite{JX3} (Lemma~1.2) that independence implies order independence. We refer to \cite{JX3} for extensive studies and  examples on  (order) independent
sequences. Below we will simply write $\E$ for $\E_{\cal{N}}$ and $\E_n$ for $\E_{VN(\cal{A}_1,\dots,
\cal{A}_{n})}$.

 It is important to observe that if  $(\cal{A}_n)_{n\geq 1}$ is  an  order independent
sequence of von Neumann subalgebras and $a_n \in L_p(\cal{A}_n)$
with $\E(a_n)=0$ ($n\geq 1$) then $(a_n)_{n\geq 1}$ is a martingale
difference sequence  with respect to the increasing filtration $(VN(\cal{A}_1,\dots,
\cal{A}_{n}))_{n\geq 1}$. In this case, one clearly see  from the definition that when $p\geq 2$, the row and column conditioned square functions take the following simpler forms:
\[
s_c\big(\sum_{n\geq 1} a_n\big)=\Big( \sum_{n\geq 1}\E(a_n^*a_n) \Big)^{1/2}\ \text{and} \  s_r\big(\sum_{n\geq 1} a_n\big)=\Big( \sum_{n\geq 1}\E(a_na_n^*)\Big)^{1/2}.
\]
For the remaining of this section, any reference to martingales is understood to be with respect to the filtration described  above.

From the preceding discussion,  for the special case of sums of order independent sequences,   Theorem~\ref{pl2} reads as follows:
\begin{corollary}\label{Rosbig2}
Let $\Phi$ be an Orlicz function satisfying $2<p_\Phi \leq q_\Phi <\infty$. 
There exist  positive constants $\delta_\Phi$  and $\eta_\Phi$ depending only on $\Phi$ such that  for every  order independent sequence $(a_n)_{n\geq 1} \subset   L_\Phi(\M)$ with $\E(a_n)=0$, the following inequalities hold:
\begin{equation*}
\delta_\Phi^{-1} {M}_\Phi(a) \leq \T\big[\Phi(\big| \sum_{n\geq 1} a_n \big|)\big] \leq \eta_\Phi  M_\Phi(a) 
\end{equation*}
where $\displaystyle{{M}_\Phi(a)=\max\Big\{ \sum_{n\geq 1}\T\big[ \Phi\big( |a_n|\big) \big], \T\Big[ \Phi\big(  \big(\sum_{n\geq 1} \E(a_n^*a_n)\big)^{1/2} \big)\Big], \T\Big[ \Phi\big(  \big(\sum_{n\geq 1}\E(a_n a_n^*)\big)^{1/2} \big)\Big]\Big\}}$.
\end{corollary}

However, when $1<p_\Phi\leq q_\Phi<2$, the case of sums of independent sequences can not be  directly  read from Theorem~\ref{main} since the decomposition we have in  the statement  of Theorem~\ref{main} are not necessarily made up of independent sequences. Handling  this case requires  a  way of modifying martingale difference sequences into independent sequences.   Below, we adapt the approach of \cite{JX3} for this reduction. In order to state our results,  we need to  formally introduce  some new notation.

For any finite sequence $(a_n)_{1\leq n \leq N} \in \mathcal{F}$, we  define
\[
\big\| (a_n)_{1\leq n \leq N}\big\|_{L_\Phi(\M,\E,\ell_2^c)} :=\Big\| \Big( \sum_{n=1}^N \E(a_k^*a_k) \Big)^{1/2} \Big\|_{L_\Phi(\M)}.
\]
If  we set $\overline{a}=\sum_{n=1}^N  e_{1,n} \otimes a_n \in L_1(B(\ell_2^N) \otimes \M) \cap ( B(\ell_2^N) \otimes \M)$  and 
$\widetilde{ \E}=Id \otimes \E$, then  we have
\[
\big\| (a_n)_{1\leq n \leq N}\big\|_{L_\Phi(\M,\E,\ell_2^c)}= \big\| \overline{a}\big\|_{L_\Phi(B(\ell_2^N) \otimes \M, \widetilde{\E})}
\]
where $L_\Phi(B(\ell_2^N)  \otimes \M, \widetilde{\E}) $ is the conditioned space introduced in Section~2. Therefore, $\| \cdot\|_{L_\Phi(\M,\E,\ell_2^c)}$ defines a norm on the linear space $\mathcal{F}$. We define  $L_\Phi(\M,\E,\ell_2^c)$ to be  the completion of the space $(\mathcal{F}, \| \cdot\|_{L_\Phi(\M,\E,\ell_2^c)})$. The space $L_\Phi(\M,\E,\ell_2^r)$ is defined in a similar way.

 Now,  we consider the subspace $\mathcal{F}^{({\rm Ind})}$ of  $\mathcal{F}$ consisting of all sequences $(a_n)_{n\geq 1}$ in $\mathcal{F}$  such that $a_n \in L_1(\mathcal{A}_n) \cap \mathcal{A}_n$ and $\E(a_n)=0$ and let 
 $\cal{R}_\Phi^c(\M)$ be the closure of $\mathcal{F}^{({\rm Ind})}$ in   $L_\Phi(\M,\E,\ell_2^c)$.  Similarly, we may define the corresponding subspaces of $L_\Phi(\M,\E,\ell_2^r)$ and $L_\Phi(\M\overline{\otimes} \ell_\infty)$ which will  be denoted by $\cal{R}_\Phi^r(\M)$ and $\cal{R}_\Phi^d(\M)$, respectively. When $\Phi(t)=t^p$,  these are exactly the spaces $\cal{R}_p^c(\M)$, $\cal{R}_p^r(\M)$, and $\cal{R}_p^d(\M)$ introduced  in \cite{JX, JX3}. If we denote by $\mathrm{J}: \mathcal{F}^{({\rm Ind})} \to \mathcal{F}_M$ the map defined by  $(a_n)_{n\geq 1} \mapsto \sum_{n\geq  1} a_n$, then for $s \in \{d,c,r\}$,  it extends to an isometric embedding $\mathrm{J}_\Phi^s: \cal{R}_\Phi^s(\M) \to \h_\Phi^s(\M)$. Next, we consider the linear map
 $
 \varTheta: \mathcal{F}_M  \to  \mathcal{F}^{({\rm Ind})}$ defined by setting for any given $x \in  \mathcal{F}_M$, 
 \begin{equation}\label{projection1}
 \varTheta(x)_n:=\begin{cases}
0 \quad &\text{if $n=1$},\\
\displaystyle{\E_{\cal{A}_n}(dx_n)  }  \quad &\text{if $n\geq 2$.}
\end{cases}  
 \end{equation}
It is clear that for every $n\geq 1$, $\E(\varTheta(x)_n)=0$ and therefore $\varTheta(x) \in   \mathcal{F}^{({\rm Ind})}$. The following result is our main tool 
in the proof of Theorem~\ref{Rosless2} below.
\begin{proposition}\label{main-projection}  Let $\Phi$ be an Orlicz function with $1<p_\Phi \leq q_\Phi <\infty$. Then for $s\in \{d,c,r\}$,  $\varTheta: \h_\Phi^s(\M) \to \cal{R}_\Phi^s(\M)$ is bounded. Moreover, 
there exists a constant $C_\Phi$ such that for every $x \in \h_\Phi^d(\M)$ (respectively, $y \in \h_\Phi^c(\M)$), 
\[
\sum_{n\geq 2} \T\big[ \Phi\big(|\varTheta(x)_n|\big)\big] \leq C_\Phi  \sum_{n\geq 1}\T\big[ \Phi\big(|dx_n|\big)\big],
\]
respectively,
\[
\T\Big[ \Phi\Big( s_c\big(\sum_{n\geq 2} \varTheta(y)_n \big) \Big) \Big] \leq  C_\Phi
\T\Big[ \Phi\big(s_c(y)\big)\Big].
\]
\end{proposition}
We begin with the  verification of the following particular case: 
\begin{lemma}\label{lem:projection1}
Let $1<p<\infty$. Then  for $s\in \{d,c,r\}$,   $\varTheta: \h_p^s(\M) \to \cal{R}_p^s(\M)$ is a  contraction.
\end{lemma}
\begin{proof}
The diagonal part is trivial from the boundedness of conditional expectations in $L_p(\M)$ so it suffices to verify the statement for the column version. We use the fact that $L_p(\M,\E,\ell_2^c)^*=L_{p'}(\M,\E,\ell_2^c)$ where $p'$ denotes the index conjugate to $p$ (see \cite[Lemma~0.1]{JX3}).
Let $x \in \h_p^c(\M)$ and  fix  a sequence $(v_n)$ from  the unit ball of $ L_{p'}(\M,\E,\ell_2^c)$ (with $v_1=0$) so that
\[
\big\|\varTheta(x)\big\|_{\cal{R}_p^c}=\big\|\varTheta(x)\big\|_{L_p(\M,\E,\ell_2^c)}= \sum_{n\geq 2}\T(\varTheta(x)_n v_n^*)=\sum_{n\geq 2}\T(\E_{\cal{A}_n}(dx_n)v_n^*).
\]
By trace invariance and duality between $\h_p^c(\M)$ and $\h_{p'}^c(\M)$, we have
\begin{align*}
\big\|\varTheta(x)\big\|_{\cal{R}_p^c} &= \sum_{n\geq 2} \T\big( dx_n\big[ \E_{\cal{A}_n}(v_n^*) -\E(v_n^*)\big]\big)\\
&\leq \Big\|\sum_{n\geq 2} dx_n \Big\|_{\h_p^c} . \Big\| \sum_{n\geq 2} \E_{\cal{A}_n}(v_n) -\E(v_n)\Big\|_{\h_{p'}^c}.
\end{align*}
One can easily see by  independence  that  for any $n\geq 2$,  the following holds:
\[
\E_{n-1}|\E_{\cal{A}_n}(v_n)-\E(v_n)|^2=\E\big[ \E_{\cal{A}_n}(v_n)^*\E_{\cal{A}_n}(v_n)\big]- \E(v_n)^* \E(v_n) \leq \E(v_n^*v_n).
\]
This implies in particular that  
\[\Big\| \sum_{n\geq 2} \E_{\cal{A}_n}(v_n) -\E(v_n)\Big\|_{\h_{p'}^c} \leq \big\| (v_n) \big\|_{L_{p'}(\M,\E,\ell_2^c)} \leq 1.
\]
We deduce that 
$
\big\|\varTheta(x)\big\|_{\cal{R}_p^c} \leq \Big\|\sum_{n\geq 2} dx_n \Big\|_{\h_p^c} \leq \big\|x \big\|_{\h_p^c}
$
proving that $\varTheta$ is a contraction.
\end{proof}
\begin{remark}
In the proof of Lemma~\ref{lem:projection1}, it is crucial that $v_1=0$. Otherwise, we only get  the equality $\E_{n-1}|\E_{\cal{A}_n}(v_n)-\E(v_n)|^2=|v_1-\E(v_1)|^2$  when  $n=1$. As a result,  the estimate $\big\| \sum_{n\geq 1} \E_{\cal{A}_n}(v_n)-\E(v_n)\big\|_{\h_{p'}^c} \leq  \big\| (v_n) \big\|_{L_{p'}(\M,\E,\ell_2^c)}$ would not be achieved.  This is the primary reason for choosing  $\varTheta(x)_1=0$ in the definition of $\varTheta$.
\end{remark}
The proof of Proposition~\ref{main-projection} is now a simple interpolation of Lemma~\ref{lem:projection1} together with Proposition~\ref{h-interpolation}. We leave the details to the reader. \qed

\medskip

The next theorem is our main result for this section.    It is a  $\Phi$-moment generalization  of  the  noncommutative Rosenthal inequalities from \cite[Theorem~3.2]{JX3}.
\begin{theorem}\label{Rosless2} Let $\Phi$ be an Orlicz function satisfying $1<p_\Phi \leq q_\Phi <2$.
There exist  positive constants ${\tilde{\delta}_\Phi}$  and $\tilde{\eta}_\Phi$ depending only on $\Phi$ such that  for every order independent sequence $(x_n)_{n\geq 1}  \subset  L_\Phi(\M)$ with $\E(x_n)=0$, the following inequalities hold:
\begin{equation*}
\tilde{\delta}_\Phi^{-1} \widetilde{S}_\Phi(x) \leq \T\Big[\Phi\Big(\Big|\sum_{n\geq 1} x_n\Big|\Big)\Big] \leq  \tilde{\eta}_\Phi \widetilde{S}_\Phi(x) 
\end{equation*}
where $\displaystyle{\widetilde{S}_\Phi(x)=\inf\Big\{ \T\Big[ \Phi \big(s_c\big(\sum_{n\geq 1} x_n^c\big)\big)\Big] +  \T\Big[ \Phi\big(\big(s_r\big(\sum_{n\geq 1} x_n^r\big) \big)\Big]  + \sum_{n\geq 1} \T\big[\Phi(|x_n^d|)\big] \Big\}}$
 with  the infimum  being taken over all  $(x_n^c) \in \cal{R}_\Phi^c(\M)$, $(x_n^r) \in \cal{R}_\Phi^r(\M)$, and $(x_n^d) \in \cal{R}_\Phi^d(\M)$ such that  for every $n\geq 1$, $x_n= x_n^c + x_n^r + x_n^d$.
\end{theorem}
\begin{proof}
Since  for $s\in\{d, c, r\}$, the map $\mathrm{J}_\Phi^s: \cal{R}_\Phi^s(\M) \to \h_\Phi^s(\M)$  is an isometric embedding, it is clear that   $S_\Phi(x) \leq \widetilde{S}_\Phi(x)$.  Thus, the second inequality follows immediately from Theorem~\ref{main}. 

The proof of the first inequality is a combination of Theorem~\ref{main} and Proposition~\ref{main-projection}. First, by Theorem~\ref{main}, there exists a constant $\delta_\Phi$ such that if $(x_n)$ is as in the statement of the theorem then   there exists a decomposition $x_n=d\alpha_n + d\beta_n + d\gamma_n$  where   $\alpha  \in \h_\Phi^d(\M)$, $\beta  \in \h_\Phi^c(\M)$,  $\gamma \in \h_\Phi^r(\M)$, and 
\[
\T\Big[ \Phi\big(s_c(\beta)\big)\Big] + \T\Big[ \Phi\big(s_r(\gamma)\big)\Big] +
\sum_{n\geq 1}  \T\big[\Phi(|d\alpha_n|)\big]
\leq 2\delta_\Phi  \T\Big[\Phi\Big(\Big|\sum_{n\geq 1} x_n\Big|\Big)\Big].
\]
Let $x_1^d=x_1$ and $x_n^d= \varTheta(\alpha)_n$ for $n\geq 2$. Similarly,  let
$x^c=\varTheta(\beta)$ and $x^r=\varTheta(\gamma)$.  Then for every $n\geq 1$, $x_n =x_n^d+ x_n^r +x_n^r$. From Proposition~\ref{main-projection}, 
$x^d \in \cal{R}_\Phi^d(\M)$, $x^c \in \cal{R}_\Phi^c(\M)$, and $x^r \in \cal{R}_\Phi^r(\M)$. Moreover, there exists a constant $C_\Phi$ such that:
\begin{align*}
\T\Big[ &\Phi\Big( s_c\big(\sum_{n\geq 2} x_n^c \big) \Big) \Big]  +
\T\Big[ \Phi\Big( s_r\big(\sum_{n\geq 2} x_n^r \big) \Big) \Big]  +
\sum_{n\geq 1} \T\big[\Phi(|x_n^d|)\big]\\
 & \leq   \T\Big[ \Phi(|x_1|) \Big] + C_\Phi\Big\{\T\Big[ \Phi\big( s_c(\beta) \big) \Big]   +    \T\Big[ \Phi\big( s_r(\gamma) \big) \Big]   +  \sum_{n\geq 1} \T\big[\Phi(|d\alpha_n|)\big]\Big\}\\
 &\leq \T\Big[ \Phi(|x_1|) \Big] + 2C_\Phi \delta_\Phi  \T\Big[\Phi\Big(\Big|\sum_{n\geq 1} x_n\Big|\Big)\Big].
 \end{align*}
 Since $\E_{\mathcal{A}_1}$ is bounded in $L_\Phi(\M)$, we have $\T\Big[ \Phi(|x_1|) \Big] \leq D_\Phi \T\Big[\Phi\Big(\Big|\sum_{n\geq 1} x_n\Big|\Big)\Big]$ for some constant $D_\Phi$. We conclude that 
 \[
\T\Big[ \Phi\Big( s_c\big(\sum_{n\geq 2} x_n^c \big) \Big) \Big]  +
\T\Big[ \Phi\Big( s_r\big(\sum_{n\geq 2} x_n^r \big) \Big) \Big]  + \sum_{n\geq 1} \T\big[\Phi(|x_n^d|)\big]
 \leq (D_\Phi +2C_\Phi\delta_\Phi)\T\Big[\Phi\Big(\Big|\sum_{n\geq 1} x_n\Big|\Big)\Big].
\]
This completes the proof.
\end{proof}
We now consider the  noncommutative Rosenthal inequalities for case of noncommutative symmetric spaces. Following \cite{RW}, we  let $E$ denote a symmetric space on $(0,\infty)$  that satisfies the Fatou property. We denote by $p_E$ and $q_E$  the lower and upper Boyd indices respectively.
 We may repeat verbatim the construction above  and define  the spaces  $\mathcal{R}_E^c(\M)$, $\mathcal{R}_E^r(\M)$, and $\mathcal{R}_E^d(\M)$
by simply replacing $L_\Phi$ with $E$. Obvious modification of  the proof of Proposition~\ref{main-projection}   also gives that $\varTheta$ extends to a bounded linear map from $\h_E^s(\M)$ into $\mathcal{R}_E^s(\M)$ for $s \in \{d,c,r\}$. Combining this result with the  extension of  the Burkholder inequalities  to noncommutative symmetric space from \cite[Theorem~3.1]{RW}, we may also state the  following generalization of \cite[Theorem~3.2]{JX3}:
\begin{theorem}\label{ind-sym} Let $E$ be a symmetric function space defined on $(0,\infty)$ with the Fatou property and assume that $1<p_E \leq q_E<2$. Let $x_n \in E(\cal{A}_n)$
such that $\E(x_n)=0$. Then
\[
\Big\| \sum_{n\geq 1} x_n \Big\|_{E(\M)} \simeq_E \inf\left\{ \Big\| (x_n^d) \Big\|_{\cal{R}_E^d} + \Big\| (x_n^c) \Big\|_{\cal{R}_E^c}  +\Big\| (x_n^r) \Big\|_{\cal{R}_E^r} \right\} 
\]
where the infimum is taken over all decomposition $x_n =x_n^d + x_n^c + x_n^r$ with $(x_n^d) \in \cal{R}_E^d(\M)$, $(x_n^c) \in \cal{R}_E^c(\M)$, and $(x_n^r) \in \cal{R}_E^r(\M)$.
\end{theorem}
A version of Theorem~\ref{ind-sym} for the case where the Boyd indices satisfy the condition $2<p_E\leq q_E <\infty$ was first obtained in \cite[Theorem~6.3]{Dirk-Pag-Pot-Suk}. Similar line of result for  martingale BMO-norms of sums of noncommuting independent sequences were also considered in \cite[Theorem~5.3]{Ran21}.

\smallskip

As  illustrations, we  observe that  all examples  treated in \cite[Section~3]{JX3} can be easily adapted to  Corollary~\ref{Rosbig2}, Theorem~\ref{Rosless2}, and Theorem~\ref{ind-sym}. As a sample result,  the  state  the $\Phi$-moment generalization of \cite[Theorem~3.3]{JX3}:
\begin{theorem}\label{random} Let $\Phi$ be an Orlicz function and $(x_{ij})$  be a finite matrix with entries in $L_\Phi(\M)$. Assume that the $x_{ij}$'s are independent with respect to $\E$ and $\E(x_{ij})=0$. Then 

$\bullet$ for $1<p_\Phi \leq q_\Phi<2$,
\begin{equation*}
\begin{split}
\T &\otimes\tr \Big[ \Phi\big( \big|\sum_{ij}  x_{ij} \otimes e_{ij}\big| \big) \Big] 
\simeq_\Phi \\
&\inf\left\{\sum_{ij}\T\big[ \Phi\big(|x_{ij}^d|\big)\big] + \sum_j \T\Big[ \Phi\big( \big[ \sum_i \E(|x_{ij}^c|^2) \big]^{1/2} \big)\Big] +  \sum_i \T\Big[ \Phi\big( \big[ \sum_j \E({|x_{ij}^r}^*|^2) \big]^{1/2} \big)\Big]\right\}
\end{split} 
\end{equation*}
where the infimum is taken over all decompositions $x_{ij} =x_{ij}^d + x_{ij}^c +x_{ij}^r$ with  mean zero elements $x_{ij}^d$, $x_{ij}^c$, and $x_{ij}^r$, which, for each couple $(i,j)$, belong  to the  Orlicz space associated with  the von Neumann algebra generated by $x_{ij}$.

$\bullet$ for $2<p_\Phi \leq q_\Phi<\infty$, 
\begin{equation*}
\begin{split}
\T&\otimes\tr \Big[ \Phi\big( \big|\sum_{ij}  x_{ij} \otimes e_{ij}\big| \big) \Big] 
\simeq_\Phi \\
&\max\left\{\sum_{ij}\T\big[ \Phi\big(|x_{ij}|\big)\big], \sum_j \T\Big[ \Phi\big( \big[ \sum_i \E(|x_{ij}|^2) \big]^{1/2} \big)\Big], \sum_i \T\Big[ \Phi\big( \big[ \sum_j \E(|x_{ij}^*|^2) \big]^{1/2} \big)\Big]\right\}.
\end{split} 
\end{equation*}
\end{theorem}

As in the case of $p$-th moments,  if  the von algebra $\M$ is  taken to be the $L_\infty$-space defined on  a probability space,  then Theorem~\ref{random} becomes  $\Phi$-moments  inequalities of  random matrices.

\medskip

We conclude this section by noting that  by applying the  reduction technique used above  to the simultaneous decomposition stated in  Theorem~\ref{simultaneous}, we may also   achieve the following version   for independent sequences:

\begin{proposition}\label{simultaneous2} There exists a family  of constants $\{\kappa_p' : 1<p<2\} \subset \mathbb{R}_+$  satisfying the following:
 if  $(x_n)_{n\geq 1}$ is an order independent sequence  in $L_1(\M) \cap L_2(\M)$ with $\E(x_n)=0$ for all $n\geq 1$, then there exist three independent sequences $(a_n)_{n\geq 1}\in \cap_{1<p<2} \cal{R}_p^d(\M)$, $(b_n)_{n\geq 1}\in \cap_{1<p<2} \cal{R}_p^c(\M)$, and $(c_n)_{n\geq 1}\in \cap_{1<p<2}\cal{R}_p^r(\M)$
   such that:
\begin{enumerate}[{\rm(i)}]
\item  $x_n=a_n +b_n + c_n$,\ $n\geq 1$,
\item for every $1<p< 2$, the following inequality holds:
\[
\big\| (a_n) \big\|_{\cal{R}^d_p}
 + \big\| (b_n) \big\|_{\cal{R}_p^c} + \big\| (c_n)
\big\|_{\cal{R}_p^r}\leq \kappa_p' \big\| \sum_{n\geq 1} x_n \big\|_p.
\]
\end{enumerate}
\end{proposition}
\section{Concluding remarks}

We begin with the following $\Phi$-moment analogue of the Burkholder-Gundy inequalities due to Dirksen and Ricard:
\begin{proposition}[{\cite[Corollary~3.3]{Dirksen-Ricard}}]\label{prop:BG}
If $1<p_\Phi \leq q_\Phi<\infty$,  then there exist a constant $C_\Phi$,  depending only on $\Phi$, such that  for any $x \in L_\Phi(\M)$,
\[
\T\big[ \Phi\big( |x|\big)\big] \leq C_\Phi \max\Big\{ \T\big[\Phi\big( \big(\sum_{n\geq 1} |dx_k|^2\big)^{1/2}\big)\big] , \T\big[\Phi\big( \big(\sum_{n\geq 1} |dx_k^*|^2\big)^{1/2}\big)\big] \Big\}.
\]
\end{proposition}
A natural question  that arises  is whether a conditioned  version   of the above result holds. More precisely, we may ask the following problem: 
\begin{problem}
Does the second inequality in Theorem~\ref{pl2} hold under the weaker condition $1<p_\Phi \leq q_\Phi<\infty$?
\end{problem}
One may also consider the dual question: does the first inequality  in Theorem~\ref{main}  remain  valid if we only assume that $1<p_\Phi \leq q_\Phi <\infty$? 
These questions are still open even for the particular cases of independent sequences.
We should note here that the restriction $q_\Phi<2$ in Theorem~\ref{main} is due to our use of the simultaneous decompositions stated in  Theorem~\ref{simultaneous}.

\medskip

 In light of recent developments on theory of  noncommutative maximal functions, it would be desirable to have the exact noncommutative analogue of  \eqref{Burkholder} by replacing the diagonal term $\sum_{n\geq 1} \T\big[ \Phi\big(|dx_n|\big)\big]$  in  the statement of Theorem~\ref{pl2} by   an appropriate  \lq\lq $\Phi$-moment\rq\rq  maximal  function term.  Such  noncommutative maximal functions associated with Orlicz functions were already considered in \cite[Definition~3.2]{Bekjan-Chen-Ose} as follows:
\[
\T\big[ \Phi\big({\sup_n}^+ dx_n\big) \big] :=\inf\Big\{ \frac{1}{2}\Big( \T\big[ \Phi\big(|a|^2\big)\big] + \T\big[ \Phi\big(|b|^2\big)\big]\Big) \sup_n \|y_n\|_\infty \Big\}
\]
where the infimum is taken over all decompositions $dx_n=ay_n b$ for $a,b \in L_0(\M)$ and $(y_n) \subset \M$ with $|a|^2, |b|^2 \in L_\Phi(\M)$  and $\sup_n\|y_n\|_\infty \leq 1$.  The following problem is still open.
\begin{problem}
Assume that $2<p_\Phi \leq q_\Phi <\infty$ and $x \in L_\Phi(\M)$. Do we have
\[
\T\big[ \Phi\big(|x|\big)\big] \simeq_\Phi \max\Big\{ \T\big[ \Phi\big({\sup_n}^+ dx_n\big) \big], \T\big[ \Phi( s_c( x))\big],\T\big[ \Phi( s_r( x))\big] \Big\}?
\]
\end{problem}
As shown in \cite{JX3}, the answer  to the above problem is positive for  the case of  $p$-th moments when  $p \geq 2$. By duality, the corresponding result  involving $\ell_1$-valued  noncommutative $L_p$-spaces is also known for the case  $1<p<2$.
A first step toward this direction would be to improve the simultaneous decomposition stated in Theorem~\ref{simultaneous}  by replacing the diagonal term  $\|a\|_{\h_p^d}$ by  $\|(da_n) \|_{L_p(\M;\ell_1)}$. We refer to \cite{JX3} for the formal definition of the space $L_p(\M;\ell_1)$.

\medskip

We conclude by   noting that   the noncommutative Burkholder inequalities are valid for martingales in $L_p$-spaces associated with type III von Neumann algebras (\cite{JX}). In  \cite{Labu},   a theory of Orlicz spaces has been  developed  for  type III von Neumann algebras in the  spirit of  the construction  of  the Haagerup $L_p$-spaces. An interesting future direction  would be to develop a $\Phi$-moment theory for  the type III-case  using \cite{Labu}.

\def\cprime{$'$}
\providecommand{\href}[2]{#2}

\end{document}